\documentclass[11pt,reqno,a4paper]{amsart}
\usepackage{amsmath,amsthm,verbatim,amscd,amssymb,setspace,hyperref,esint,mathtools}
\usepackage{exscale,color}

\usepackage{mathrsfs}

\usepackage{enumitem}

\usepackage[colorinlistoftodos,prependcaption]{todonotes}

\renewcommand{\epsilon}{\varepsilon}
\newcommand{\N}{\mathbb{N}}

\newcommand{\R}{\mathbb{R}}
\newcommand{\C}{\mathbb{C}}

\renewcommand{\Re}{\operatorname{Re}}

\newcounter{mtheorem}
\newtheorem{mtheorem}[mtheorem]{Theorem}

\newtheorem{mcor}[mtheorem]{Corollary}

\newcommand{{\vol}}{\rm vol}
\newcommand{\p}{\partial}

\newcommand{\Ric}{\operatorname{Ric}}

\newcommand{\Rm}{\operatorname{Rm}}

\def\tr{\operatorname{tr}}

\def\Id{\operatorname{Id}}

\def \bp {\bar{\partial}}

\def\tr{\operatorname{tr}}
\def\Ric{\operatorname{Ric}}
\def\tr{\operatorname{tr}}

\def\Id{\operatorname{Id}}

\def\vol{\operatorname{vol}}

\newtheoremstyle{fancy}{}{}{\itshape}{}{\textbf\bgroup}{.\egroup}{ }{}
\newtheoremstyle{fancy2}{}{}{\rm}{}{\textbf\bgroup}{.\egroup}{ }{}

\theoremstyle{fancy}
\newtheorem{theorem}{Theorem}[section]
\newtheorem{lemma}[theorem]{Lemma}
\newtheorem{corollary}[theorem]{Corollary}

\newtheorem{prop}[theorem]{Proposition}

\theoremstyle{fancy2}
\newtheorem{definition}[theorem]{Definition}

\newtheorem{remark}[theorem]{Remark}

\newtheorem{claim}[theorem]{Claim}

\setlist{leftmargin=*}

\numberwithin{equation}{section}
\begin{document}
\title{Uniqueness of shrinking K\"ahler-Ricci solitons on resolutions of K\"ahler cones}
\date{\today}

\author{Ronan J.~Conlon}
\address{Department of Mathematical Sciences, The University of Texas at Dallas, Richardson, TX 75080}
\email{ronan.conlon@utdallas.edu}
\author{Alix Deruelle}
\address{Universit\'e Paris-Saclay, CNRS, Laboratoire de math\'ematiques d'Orsay, 91405, Orsay, France}
\email{alix.deruelle@universite-paris-saclay.fr}

\date{\today}

\begin{abstract}
We show that any complete shrinking gradient K\"ahler-Ricci soliton on a resolution of a K\"ahler cone is necessarily asymptotically conical.
From a result of Esparza, it then follows that up to pullback by biholomorphism, there exists at most one complete shrinking gradient K\"ahler-Ricci soliton on
such a resolution. This confirms a special case of the uniqueness part of a conjecture of Song-Zhang. Some other consequences are also discussed.
\end{abstract}

%We construct a new example of a two-dimensional shrinking gradient K\"ahler-Ricci soliton with bounded scalar curvature and complete
%the classification of such manifolds.

\maketitle

\markboth{Ronan J.~Conlon, and Alix Deruelle}{Uniqueness of shrinking K\"ahler-Ricci solitons on resolutions of K\"ahler cones}

\tableofcontents

\section{Introduction}

\subsection{Overview}
A \emph{Ricci soliton} is a triple $(M,\,g,\,X)$, where $M$ is a Riemannian manifold endowed with a complete Riemannian metric $g$
and a complete vector field $X$, such that
\begin{equation}\label{hot}
\Ric(g)+\frac{1}{2}\mathcal{L}_{X}g=\lambda g
\end{equation}
for some $\lambda\in\mathbb{R}$. The vector field $X$ is called the
\emph{soliton vector field}. If $X=\nabla^{g} f$ for some smooth real-valued function $f$ on $M$,
then we say that $(M,\,g,\,X)$ is \emph{gradient}. In this case, the soliton equation \eqref{hot}
becomes $$\Ric(g)+\operatorname{Hess}_{g}(f)=\lambda g,$$
and we call $f$ the \emph{soliton potential}. In the case of gradient Ricci solitons, the completeness of $X$ is guaranteed by the completeness of $g$
\cite{zhang12}.

Let $(M,\,g,\,X)$ be a Ricci soliton. If $g$ is K\"ahler with K\"ahler form $\omega$ and $X$ is real holomorphic, then we say that $(M,\,g,\,X)$ 
(or $(M,\,\omega,\,X))$ is a \emph{K\"ahler-Ricci soliton}. If $(M,\,g,\,X)$ is in addition gradient, then \eqref{hot} may be rewritten as
\begin{equation}\label{krseqn}
\rho_{\omega}+i\partial\bar{\partial}f=\lambda\omega,
\end{equation}
where $\rho_{\omega}$ is the Ricci form of $\omega$ and $f$ is the soliton potential.

Finally, a Ricci soliton and a K\"ahler-Ricci soliton are called \emph{steady} if $\lambda=0$, \emph{expanding}
if $\lambda<0$, and \emph{shrinking} if $\lambda>0$ in \eqref{hot}.
One can always normalise $\lambda$, when non-zero, to satisfy $|\lambda|=1$. We henceforth assume that this is the case.

Ricci solitons are interesting both from the point of view of canonical metrics and of the Ricci flow. On one hand, they represent one direction in which the notion of an Einstein manifold can be generalised. On compact manifolds, shrinking Ricci solitons are known to exist in several instances where there are obstructions to the existence of Einstein metrics; see for example \cite{soliton}. By the maximum principle, there are no nontrivial expanding or steady Ricci solitons on compact manifolds. However, there are many examples on noncompact manifolds; see for example \cite{Cao-KR-sol, FIK, Wang, con-der, schafer2, schaffer, heather, conlon33} among others. On the other hand, one can associate to a Ricci soliton a self-similar solution of the Ricci flow, and gradient shrinking Ricci solitons in particular provide models for finite-time Type I singularities of the flow \cite{topping, naber}. From this perspective, it is an important problem to classify and construct examples of such solitons in order to better understand singularity formation in the Ricci flow.

In this article, we study the uniqueness of complete shrinking gradient Kähler-Ricci solitons on anti-canonically polarised resolutions of Kähler cones. We show that any such soliton on such a resolution must be asymptotically conical. It then follows from a result of Esparza \cite{Esp25} that there exists at most one soliton of this type on a given resolution. 
In addition, our result allows us to remove the assumption of bounded Ricci curvature from the classification results of \cite[Theorem E]{cds}, thereby demonstrating that
up to pullback by an element of $GL(n, \mathbb{C})$, the only complete shrinking gradient K\"ahler-Ricci soliton on $\mathbb{C}^n$ is the flat Gaussian shrinking soliton and on the total space of $\mathcal{O}(-k) \to \mathbb{P}^{n-1}$ for $0 < k < n$ is the $U(n)$-invariant example of Feldman-Ilmanen-Knopf \cite{FIK}.

\subsection{Main results}

We first show that the asymptotic geometry of the shrinking gradient K\"ahler-Ricci soliton is determined by the complex geometry. 
\begin{mtheorem}[Asymptotics of shrinkers on resolutions of cones]\label{classify1}
%Let $C_{0}$ be a K\"ahler cone with complex structure $J_{0}$ and radial function $r$, and let 
%$\pi:M\to C_{0}$ be a resolution of $C_{0}$ that is equivariant with respect to
%the real holomorphic torus action on $C_{0}$ generated by the flow of $J_{0}r\partial_{r}$
%and such that $-K_{M}$ is $\pi$-ample. Then 
Every complete shrinking gradient K\"ahler-Ricci soliton on a resolution $\pi:M\to C_{0}$ of a K\"ahler cone $C_{0}$ has quadratic curvature decay, or equivalently, is asymptotically conical.
\end{mtheorem}
\noindent The precise structure of asymptotically conical shrinking Kähler-Ricci solitons, together with the equivalence to the condition of quadratic curvature decay, may be found in \cite[Theorem A]{cds}. The resolution $\pi:M\to C_{0}$ defines what is known as a ``polarised Fano fibration'' in the sense of \cite[Definition 2.5]{JunshengSong}. It follows that $M$ is quasi-projective.

Although Theorem \ref{classify1} asserts that the shrinking soliton is asymptotically conical, 
one cannot a priori determine the asymptotic cone, reflecting the fact that, as noted in \cite{JunshengSong}, constructing 
shrinking gradient K\"ahler-Ricci solitons is a free boundary problem. The conical asymptotics are detected through the 
quadratic curvature decay of the shrinking soliton. Note that it is already known that the curvature of complete shrinking gradient K\"ahler-Ricci soliton surfaces is bounded \cite{wangli}.

From Theorem \ref{classify1}, it immediately follows from a theorem of Esparza \cite[Theorem 1.1]{Esp25} that 
shrinking gradient K\"ahler-Ricci solitons on resolutions of K\"ahler cones are unique up to pullback by biholomorphism.
\begin{mcor}[Uniqueness of shrinkers on resolutions of cones]\label{classify2}
%Let $C_{0}$ be a K\"ahler cone with complex structure $J_{0}$ and radial function $r$, and let 
%$\pi:M\to C_{0}$ be a resolution of $C_{0}$. 
%that is equivariant with respect to the real holomorphic torus action on $C_{0}$ generated by the flow of $J_{0}r\partial_{r}$ and such that $-K_{M}$ is $\pi$-ample.
Up to pullback by biholomorphism, there exists at most one complete shrinking gradient K\"ahler-Ricci soliton on a resolution $\pi:M\to C_{0}$ of a K\"ahler cone $C_{0}$.
\end{mcor}
\noindent In particular, it follows that up to pullback by biholomorphism, the flat Gaussian soliton on $\mathbb{C}^{n}$, together with the
shrinking gradient K\"ahler-Ricci soliton examples of \cite{charles1, FIK, chili, Wang} and \cite{futaki3, futaki-wang} are the only such solitons on their respective
underlying complex manifold. Subsequently, this removes the bounded Ricci curvature condition from 
the uniqueness results \cite[Corollary D]{charlie} and \cite[Theorem E]{cds}. 

It is interesting to observe that up to pullback by biholomorphism, the flat Gaussian shrinking soliton is the only 
complete shrinking gradient K\"ahler-Ricci soliton on $\mathbb{C}^{n}$. This is in stark contrast to the existence of complete Calabi-Yau metrics on $\mathbb{C}^{n}$, of which there are many; see for example \cite{cy1, cy3, cy2, cy5, cy4, cy6}. This result also provides evidence for \cite[Conjecture 6.3]{JunshengSong}.
Uniqueness of shrinking gradient K\"ahler-Ricci solitons up to pullback by biholomorphism on polarised fibrations 
is expected in general; cf.~\cite[Conjecture 6.1]{JunshengSong}. Corollary \ref{classify2} confirms the uniqueness part of this conjecture in a special case.

The results of \cite{cds, Esp-charlie} also yield the following consequences of the boundedness of the curvature given by Theorem \ref{classify1}. Note that, as explained on \cite[p.42]{Esp25},
there always exist maximal tori in the automorphism group of 
a non-compact analytic space.
\begin{mcor}[Properties of shrinkers on resolutions of cones]\label{classify3}
Let $(M,\,g,\,X)$ be a complete shrinking gradient K\"ahler-Ricci soliton with complex structure $J$ on a resolution $\pi:M\to C_{0}$ of a K\"ahler cone $C_{0}$. Then:
\begin{enumerate}[label=\textnormal{(\alph{*})}, ref=(\alph{*})]
\item There exists a maximal torus $\mathbb{T}$ in the automorphism group of $(M,\,J)$ acting isometrically on $g$ with $JX$ in its Lie algebra $\operatorname{Lie}(\mathbb{T})$. 
\item $JX$ is the unique minimiser of the weighted volume functional, a real-valued strictly convex function defined on a certain open convex cone of $\operatorname{Lie}(\mathbb{T})$.
\item $(\mathbb{T}\curvearrowright\pi:M\to C_{0},\,d\pi(JX))$ defines a K-polystable polarised Fano fibration. 
\end{enumerate}
\end{mcor}
\noindent We refer the reader to \cite[Definition 5.14]{cds} and the surrounding section for the definition and properties of the weighted volume functional, and to
 \cite[Sections 2 \& 5]{JunshengSong} for the definitions of polarised Fano fibrations and K-stability, respectively. We remark that any two maximal tori in the 
 automorphism group of $(M,\,J)$ admitting soliton vector fields in their Lie algebra are conjugate by \cite[Proposition 5.17]{Esp25}. 

Our method of proof of Theorem \ref{classify1} relies on the complex Monge-Amp\`ere equation. The key observation is that 
a K\"ahler cone metric with its radial vector field $r\partial_{r}$ defines a shrinking gradient K\"ahler-Ricci soliton up to terms of order $O(r^{-2})$.
Given the underlying complex manifold, one must first ascertain that the soliton vector field is in fact conical. This follows from the fact that the zero set of the 
soliton vector field is compact \cite[Proposition 3.5]{JunshengSong}, whereas previously in \cite{cds} assuming bounded Ricci curvature was necessary to 
achieve this fact. Once this has been established, we can construct a background metric which is asymptotically conical, hence asymptotically a shrinking gradient K\"ahler-Ricci soliton.
The next observation is that any shrinking gradient K\"ahler-Ricci soliton on the underlying complex manifold differs from our constructed background metric by $i\partial\bar{\partial}\varphi$ for some smooth real-valued function $\varphi$ that satisfies a complex Monge-Amp\`ere equation, the data of which decays at rate $O(r^{-2})$. 
A posteriori estimates for this equation yield bounds on the growth of $\varphi$ and on its derivatives, resulting in 
quadratic curvature decay as stated in Theorem \ref{classify1}. Quadratic curvature decay for shrinking Ricci solitons implies conical asymptotics \cite{munty}. Knowing this, 
the uniqueness result of Esparza \cite{Esp25} yields Corollary \ref{classify2}.

\subsection{Acknowledgements}
The first author is supported by a Simons Travel Grant. The second author is partially supported by grants from the
French National Research Agency  ANR-24-CE40-0702 (Project OrbiScaR) and the Charles Defforey Fondation-Institut de France via the project
``KRIS''. He also benefits from a Junior Chair from the Institut Universitaire de France.

\section{Preliminaries}\label{sec_preliminaries}

\subsection{Riemannian cones} For us, the definition of a Riemannian cone will take the following form.

\begin{definition}\label{cone}
Let $(S, g_{S})$
be a compact connected Riemannian manifold. The \emph{Riemannian cone} $C_{0}$ with \emph{link} $S$ is defined to be $\R^+ \times S$ with metric $g_0 = dr^2 \oplus r^2g_{S}$ up to isometry. The radius function $r$ is then characterized intrinsically as the distance from the apex in the metric completion.
\end{definition}

Suppose that we are given a Riemannian cone $(C_0,g_{0})$ as above. Let $(r,x)$ be polar coordinates on $C_{0}$, where $x\in S$, and for $t>0$, define a map
$$\nu_{t}: [1,2]\times S \ni (r,x) \mapsto (tr,x) \in [t,2t] \times S.$$ One checks that $\nu_{t}^{*}(g_{0})=t^{2}g_{0}$ and $\nu^{*}_{t}\circ\nabla^{g_0}=\nabla^{g_0}\circ\nu_{t}^{*}$, where $\nabla^{g_0}$ is the  Levi-Civita connection of $g_{0}$. From this, we deduce

\begin{lemma}\label{simple321}
Suppose that $\alpha\in\Gamma((TC_0)^{\otimes p}\otimes (T^{*}C_0)^{\otimes q})$ satisfies $\nu_{t}^{*}(\alpha)=t^{k}\alpha$ for every $t>0$ for some $k\in\R$. Then $|(\nabla^{g_0})^{\ell}\alpha|_{g_{0}}=O(r^{k+p-q-\ell})$ for all $\ell\in\N_0$.
\end{lemma}

We shall say that ``$\alpha=O(r^{\lambda})$ with $g_{0}$-derivatives'' whenever $|(\nabla^{g_0})^{k}\alpha|_{g_{0}}=O(r^{\lambda-k})$ for every $k \in \N_0$.
We will then also say that $\alpha$ has ``rate at most $\lambda$'', or sometimes, for simplicity, ``rate $\lambda$'', although it should be understood that (at least when $\alpha$ is purely polynomially behaved and does not contain any $\log$ terms) the rate of $\alpha$ is really the infimum of all $\lambda$ for which this holds.

\subsection{K{\"a}hler cones} 
We may further impose that a Riemannian cone is K\"ahler, as the next definition
demonstrates. Boyer-Galicki \cite{book:Boyer} is a comprehensive reference here.

\begin{definition}A \emph{K{\"a}hler cone} is a Riemannian cone $(C_0,g_0)$ such that $g_0$ is K{\"a}hler, together with a choice of $g_0$-parallel complex structure $J_0$. This will in fact often be unique up to sign. We then have a K{\"a}hler form $\omega_0(X,Y) = g_0(J_0X,Y)$, and $\omega_0 = \frac{i}{2}\p\bar{\p} r^2$ with respect to $J_0$.
\end{definition}

The vector field $r\partial_{r}$ on a K\"ahler cone is real holomorphic, and $J_{0}r\partial_r$ is real holomorphic and Killing \cite[Appendix A]{MSY}. This latter vector field is known as the \emph{Reeb field}.  The closure of its flow in the
isometry group of the link of the cone generates the holomorphic isometric action of a real torus on
$C_{0}$ that fixes the apex of the cone.

Every K\"ahler cone is affine algebraic.
\begin{theorem}\label{t:affine}
For every K{\"a}hler cone $(C_{0},g_0,J_0)$, the complex manifold $(C_{0},J_0)$ is isomorphic to the smooth part of a normal algebraic variety $V \subset \C^N$ with one singular point. In addition, $V$ can be taken to be invariant under a $\C^*$-action $(t, z_1,\ldots,z_N) \mapsto (t^{w_1}z_1,\ldots,t^{w_N}z_N)$ such that all $w_i $ are positive.
\end{theorem}
\noindent This can be deduced from arguments written down by van Coevering in \cite[\S 3.1]{vanC4}.

The holomorphic torus action on a K\"ahler cone leads to the notion of an \emph{equivariant resolution}.
\begin{definition}\label{equivariantt}
Let $C_0$ be a K\"ahler cone and let $\pi:M\to C_0$ be a resolution of $C_0$. We say that
$\pi:M\to C_0$ is a $G$-\emph{equivariant resolution} with respect to the holomorphic action of a Lie group $G$ on $C_{0}$ if the $G$-action lifts via $\pi$ to a holomorphic action on $M$.
\end{definition}
Such a resolution of a K\"ahler cone always exists; see \cite[Proposition 3.9.1]{kollar}. The following will prove useful.

\begin{lemma}[\protect{\textnormal{\cite[Lemma 2.4]{conlon33}}}]\label{nice}
Let $(C_{0},\,g_{0})$ be a K\"ahler cone with Reeb vector field $\xi$ and let $K\subseteq C_{0}$ be a compact subset containing the apex of $C_{0}$ such that $C_{0}\setminus K$ is connected.
If $u:C_{0}\setminus K\to\mathbb{R}$ is a smooth real-valued function defined on $C_{0}\setminus K$
that is pluriharmonic (meaning that $\partial\bar{\partial}u=0$) and invariant under the flow of
$\xi$, then $u$ is a real constant.
\end{lemma}

\subsection{Asymptotically conical shrinking K\"ahler-Ricci solitons}\label{cones2}

We also have:
\begin{definition}\label{d:ACK}
Let $(M,\,g)$ be a complete K\"ahler manifold with complex structure $J$ and let $(C_0,g_0)$ be a K\"ahler cone with a choice of $g_0$-parallel complex structure $J_0$. We call $M$ \emph{asymptotically conical} (AC) \emph{K\"ahler} with tangent cone $C_{0}$ if there exists a biholomorphism $\Phi: C_0\setminus K \to M \setminus K'$ with $K,K'$ compact, such that $\Phi^*g - g_0 = O(r^{-\epsilon})$ with $g_0$-derivatives for some $\epsilon > 0$. In particular, $(M,\,g)$ is AC with tangent cone $C_{0}$.
\end{definition}

AC K\"ahler manifolds only have one end. Indeed, they are $1$-convex by \cite[Lemma 2.15]{Conlon}, hence only have one end by \cite[p.454]{Rossi2}.
Notice that the curvature $\operatorname{Rm}_{g}$ of a complete AC K\"ahler manifold $(M,\,g)$ satisfies
\begin{equation}\label{eric}
A_{k}(g):=\sup_{x\in M}|\operatorname{Rm}_{g}|_{g}(x)d_{g}(p,\,x)^{2}<\infty,
\end{equation}
where $d_{g}(p,\,\cdot)$ denotes the distance to a fixed point $p\in M$ with respect to $g$. On the other hand, by \cite[Theorem A]{cds},
any complete shrinking gradient K\"ahler-Ricci soliton $(M,\,g)$ satisfying \eqref{eric} along its unique end \cite{munteanu}
is AC with $\varepsilon=2$ and $d\Phi(r\partial_{r})=X$ in Definition \ref{d:ACK}.

Moreover, $M$ is a quasi-projective equivariant resolution $\pi:M\to C_{0}$
of the tangent cone $C_{0}$ with respect to the real torus action on $C_{0}$ induced by the flow of the Reeb vector field \cite[Theorem A \& Proposition 2.25]{cds} with $-K_{M}$, admitting a hermitian metric with strictly positive curvature, being $\pi$-ample; cf.~\cite[Section 4.1]{cds} for the argument regarding this latter point. This is why, in contrast to the general definition of an AC K\"ahler manifold, we take the map $\Phi$ in Definition \ref{d:ACK} to be a biholomorphism. Conversely, any equivariant resolution $\pi:M\to C_{0}$ of a K\"ahler cone $C_{0}$ with respect to the aforementioned real torus action with $-K_{M}$ being $\pi$-ample is quasi-projective \cite{JunshengSong}. 

\subsection{Properties of shrinking K\"ahler-Ricci solitons}

In this section, we consider some properties of shrinking gradient K\"ahler-Ricci solitons. Throughout, we consider a shrinking gradient K\"ahler-Ricci soliton $(M,\,g,\,X)$ 
with complex structure $J$. No assumption is made on the curvature of $g$.

\subsubsection{The automorphism group}

As the following lemma demonstrates, the connected component of the identity of the Lie group of holomorphic isometries of $(M,\,J,\,g)$ that commute with the flow of $X$ is a compact Lie group. We denote this Lie group by $G_{0}^{X}$. This lemma was proved under the assumption of bounded Ricci curvature in \cite[Proposition 5.11]{cds}.

\begin{lemma}\label{class}
$G_{0}^{X}$ is a compact Lie group with respect to the compact-open topology.
\end{lemma}

\begin{proof}
By \cite[Proposition 3.5]{JunshengSong}, we know that the zero set of $X$ is compact. Therefore $G_{0}^{X}$ is indeed a Lie group by \cite[Proposition 5.9]{cds}, where the assumption of bounded Ricci curvature is no longer required because the proof of this proposition uses only \cite[Lemma 2.34]{cds} which only requires that the zero set of $X$ is compact. Compactness of $G_{0}^{X}$ with respect to the compact open topology follows from the Arzel\`a-Ascoli theorem because the level sets of $f$ are compact by properness of $f$ and are preserved by the action of $G_{0}^{X}$ as demonstrated in \cite[Proof of Proposition 5.10]{cds}.
\end{proof}

\subsubsection{The soliton vector field}

In the case that $M$ is in addition ``$1$-convex'', meaning
that $M$ carries a plurisubharmonic exhaustion function which is strictly plurisubharmonic
outside of a compact set, the fact that the zero set of $X$ is compact thanks to \cite{JunshengSong} allows us to remove the assumption of bounded scalar curvature from \cite[Proposition 2.27]{cds} as we see below. Since a $1$-convex space is in particular holomorphically convex,
$M$ in this case will admit a ``Remmert reduction'' $p:M\to M'$ \cite{Grau:62}, i.e., a proper holomorphic map $p:M\to M'$
onto a normal Stein space $M'$ with finitely many isolated singularities obtained by contracting the maximal compact analytic subset $E$ of $M$.
As a Stein space with only finitely many isolated singularities, \cite[Theorem 3.1]{SCV6} asserts that $M'$ admits an embedding
$h:M'\to\mathbb{C}^{P}$ into $\mathbb{C}^{P}$ for some $P\in\mathbb{N}$. 

\begin{prop}\label{sexxy}
Let $(M,\,g,\,X)$ be a complete shrinking gradient K\"ahler-Ricci soliton of complex dimension $n$.
Assume that $M$ is $1$-convex with maximal compact analytic subset $E$. Then:
\begin{enumerate}[label=\textnormal{(\roman{*})}, ref=(\roman{*})]
\item if $E=\emptyset$, then the zero set of $X$ comprises a single point and $M$ is biholomorphic to $\mathbb{C}^{n}$, or
\item if $E\neq\emptyset$, then the zero set of $X$ is contained in $E$ (and is compact).
\end{enumerate}
\end{prop}

\begin{proof}
Since the zero set of $X$ is compact by \cite[Proposition 3.15]{JunshengSong}, the proof of \cite[Proposition 2.27]{cds} gives the result.
\end{proof}

\section{Set-up of the complex Monge-Amp\`ere equation}\label{setupp}

In this section, $C_{0}$ will denote a K\"ahler cone with complex structure $J_{0}$ and 
$\pi:M\to C_{0}$ will be a resolution of $C_{0}$ with exceptional set $E$ and complex structure $J$, so that $\pi_{*}J=J_{0}$. 
Let $(M,\,g,\,X)$ be a shrinking gradient K\"ahler-Ricci soliton on $M$ with respect to $J$ with K\"ahler form $\omega$, with $X=\nabla f$ for a smooth real-valued function $f:M\to\mathbb{R}$
which we know is proper and bounded from below \cite{caoo}. Then $JX$ is real holomorphic and Killing \cite[Lemma 2.3.8]{fut2} on $M$.

In this section, we derive a complex Monge-Amp\`ere equation satisfied by the metric $g$. 

\subsection{Properties of $\pi:M\to C_{0}$}
We first collect together the following properties of $\pi:M\to C_{0}$. These properties may be deduced from \cite{JunshengSong}. However, we provide
a proof using the work of \cite{cds} when possible.
\begin{prop}[Properties of resolutions admitting shrinkers]\label{jenny}
\begin{enumerate}[label=\textnormal{(\roman{*})}, ref=(\roman{*})]
\item $M$ is quasi-projective.
\item $M$ is simply connected.
\item $\pi:M\to C_{0}$ is $-K_{M}$-ample.
\item There exists a torus $T$  acting holomorphically on $C_{0}$ with fixed point set the apex of $C_{0}$ such that $d\pi(JX)$ lies in the Lie algebra $\mathfrak{t}$ of $T$ and such that
the action of $T$ extends to a holomorphic isometric action on $(M,\,J,\,g)$ in such a way that makes $\pi:M\to C_{0}$ a $T$-equivariant resolution.
\item The vector field $d\pi(JX)$ on $C_{0}$ is the Reeb vector field of some $T$-invariant
K\"ahler cone metric $g_{0}$ on $C_{0}$, i.e., there exists a $T$-invariant K\"ahler cone metric $g_{0}$ on $C_{0}$ with radial function $r$ such that $r\partial_{r}=d\pi(X)$.
\end{enumerate}
\end{prop}
In the terminology of \cite{JunshengSong}, it follows that the pair $(T\curvearrowright\pi:M\to C_{0},\,d\pi(JX))$ defines a ``polarised Fano fibration''.

\begin{proof}[Proof of Proposition \ref{jenny}]
In what follows, we identify $M\setminus E$ with the complement of the apex of $C_{0}$ via $\pi$.
\begin{enumerate}[label=\textnormal{(\roman{*})}, ref=(\roman{*})]
\item This is the statement of \cite[Proposition 2.25]{cds}.
\item This follows from \cite{Esp-sc, JunshengSong}.
\item We begin with the following claim.
\begin{claim}\label{tosatti}
There exists a hermitian metric on $-K_{M}$ with curvature form $\Theta$ such that $i\Theta>0$.
\end{claim}

\begin{proof}[Proof of Claim \ref{tosatti}]
Recall that the K\"ahler form $\omega$ of the shrinking gradient K\"ahler-Ricci soliton on $M$ satisfies $\rho_{\omega}+i\partial\bar{\partial}f=\omega$.
Let $h$ denote the hermitian metric induced on $-K_{M}$ by $\omega$. Then this equation is equivalent to saying that $i$ times the curvature of the hermitian metric $e^{-f}h$ on $-K_{M}$
is precisely $\omega$. The claim now follows.
\end{proof}

We now proceed as in \cite[p.3056]{cds} with $-K_{M}$ in place of $K_{M}$.
To this end, let $\Theta$ be as in Claim \ref{tosatti}. Then the strictly positivity implies that
\begin{equation*}
\int_{V}(i\Theta)^{k}\wedge\omega^{\dim_{\mathbb{C}}V-k}>0
\end{equation*}
for all positive-dimensional irreducible algebraic subvarieties $V\subset E$ and for all integers $k$ such that $1\leq k\leq \dim_{\C}V$. Setting $k=\dim_{\mathbb{C}}V$, we then see that
\begin{equation*}
\int_{V}(i\Theta)^{\dim_{\mathbb{C}}V}>0\quad\textrm{for every irreducible algebraic subvariety $V\subset E$ of positive dimension.}
\end{equation*}
But since $M$ is quasi-projective by part (i), this is the same as saying that
\begin{equation*}
(D^{\dim_{\mathbb{C}}V}\cdot V)>0\quad\textrm{for every irreducible algebraic subvariety $V\subset E$ of positive dimension,}
\end{equation*}
where $D$ is now an anti-canonical divisor of $M$. Nakai's criterion for a mapping \cite[Corollary 1.7.9]{lazarfeld} now tells us that $-K_{M}$ is $\pi$-ample, as claimed.

\item Let $G_{0}^{X}$ denote the connected component of the identity of the Lie group of holomorphic isometries of $(M,\,J,\,g)$ that commute with the flow of $X$, a compact Lie group by
Lemma \ref{class}. Consider the flow of the real holomorphic Killing vector field $JX$ in $G_{0}^{X}$. As a compact Lie group, the closure of the flow of $JX$ in $G_{0}^{X}$ (with respect to the compact open topology) generates the effective holomorphic isometric action of a real torus $T \subset G_{0}^{X}$ on $(M,\,J,\,g)$ with Lie algebra $\mathfrak{t}$ containing $JX$
that preserves $E$ if it is non-empty by \cite[Lemma 2.6]{cds}. Since $M$ is $1$-convex by \cite[Lemma 2.15]{Conlon}, Proposition \ref{sexxy} tells us that the zero set of $X$, and correspondingly the fixed point set of $T$, comprises a single point if $E=\emptyset$ (in which case $M$ is biholomorphic to $\mathbb{C}^{n}$), and is contained in $E$ otherwise.
%Furthermore, \cite[Proposition 2.27]{cds} implies that each forward orbit of the
%negative gradient flow of $f$ converges to a point in this fixed point set.
By contracting $E$ if is non-empty, we see that the action of $T$ on $M$
induces an action of $T$ on $C_{0}$ with fixed point set the apex,
in such a way that the resolution $\pi:M\to C_{0}$ is equivariant with respect to $T$.

\item If $E=\emptyset$, then by Proposition \ref{sexxy}(i), $M$ is biholomorphic to $\mathbb{C}^{n}$ and the zero set of $X$ comprises a single point $p$. 
As the soliton potential $\tilde{f}$ is proper and bounded from below, it must achieve its minimum value at $p$.
By \cite[Proposition 6]{Bry-Kah-Sol}, one can linearise $X$ at $p$, i.e., there exist holomorphic coordinates $z=(z_{1},\ldots,z_{n})$ with $z(p)=0$ defined in a neighbourhood of $p$
with respect to which $\frac{1}{2}(X-iJX)=\sum_{k\,=\,1}^{n}\lambda_{k}z_{k}\partial_{z_{k}}$ for $\lambda_{k}\in\mathbb{R}$. Since $p$ is a minimum for $\tilde{f}$ and $X$ is the gradient of $\tilde{f}$, we see that necessarily $\lambda_{i}>0$. By flowing along $X$ and $\sum_{k\,=\,1}^{n}\lambda_{k}z_{k}\partial_{z_{k}}$, using \cite[Proposition 2.27]{cds}
and the fact that $X$ is complete \cite{zhang12}, we may globally linearise $X$, i.e., we can extend the holomorphic coordinates to the whole of $M$ so that $\frac{1}{2}(X-iJX)=\sum_{k\,=\,1}^{n}\lambda_{k}z_{k}\partial_{z_{k}}$ globally.
In these coordinates, it is clear that $JX$ defines a Reeb field. Define a function
$r:\mathbb{C}^{n} \to \mathbb{R}$ by setting $r=1$ on the Euclidean unit sphere in $\mathbb{C}^{n}$ and extending $r$ to have degree $1$ under the flow of $X$.
As shown in \cite{frank}, the form $\frac{i}{2}\partial\bar{\partial}r^{2}$ defines a K\"ahler cone metric on $\mathbb{C}^{n}$ with Reeb vector field $JX$.
Since the flow of $JX$ preserves the unit Euclidean sphere, it is clear that this cone metric is invariant under the action of $T$.

If $E\neq\emptyset$, then we know that the cone $C_{0}$ is an affine algebraic variety by Theorem \ref{t:affine}. From part (i), we also know that $M$ is quasi-projective.
Let $\mathcal{O}_{C_{0}}(C_{0})$ (respectively $\mathcal{O}_{M}(M)$) denote the global sections of the structure sheaf of $C_{0}$ (resp.~of $M$) and write $$\mathcal{O}_{C_{0}}(C_{0})=
\bigoplus_{\alpha\,\in\,\mathfrak{t}^{*}}\mathcal{H}_{\alpha}$$
for the weight decomposition under the action of $T$. Since $M$ is $1$-convex, by construction, $\pi:M\to C_{0}$ will be the Remmert reduction
of $M$. In particular, we have that $\pi^{*}\mathcal{O}_{C_{0}}(C_{0})=\mathcal{O}_{M}(M)$ by the properties of the Remmert reduction.
Since $\pi:M\to C_{0}$ is equivariant with respect to the action of $T$, we thus see that
$Y\in\mathfrak{t}$ acts with weight $\lambda$ on $h\in\mathcal{O}_{C_{0}}(C_{0})$ if and only if
it acts with weight $\lambda$ on the unique lift $\pi^{*}h$ of $h$ to $\mathcal{O}_{M}(M)$.

We now demonstrate that:
\begin{claim}\label{loveit}
$\alpha(JX)>0$ for all $\alpha\in\mathfrak{t}^{*}$ such that $\mathcal{H}_{\alpha}\neq\emptyset$ and $\alpha\neq0$.
\end{claim}

\noindent This means that $JX$ acts with 
positive weights on non-constant holomorphic functions on $C_{0}$. This will be true if and only if $JX$ acts with positive weights on 
non-constant holomorphic functions on $M$. We prove this latter statement using an argument verbatim from the proof of \cite[Theorem A.10]{cds}.

\begin{proof}[Proof of Claim \ref{loveit}]
Let $h$ be a non-constant holomorphic function on $M$ on which $JX$ acts with weight $\lambda$
so that $X(h)=\lambda h$, and let $x\in M\setminus E$ be a point where $h(x)\neq0$.
Denote by $\gamma_{x}(t)$ the flow line of $X$ with $\gamma_{x}(0)=x$. Then 
$$\frac{d}{dt}h(\gamma_{x}(t))=\lambda h(\gamma_{x}(t))$$
so that
\begin{equation}\label{star}
h(\gamma_{x}(t))=h(x)e^{-\lambda t}\quad\textrm{for all $t<0$}.
\end{equation}
Now, we know that $X=\nabla f$ and $f$ is proper and bounded from below, so from
\cite[Proposition 2.28]{cds} we see that there is a sequence $t_{i}\to-\infty$ as $i\to+\infty$ such that $(\gamma_{x}(t_i))_i$ converges to a point $x_{\infty}\in M$ satisfying $X(x_{\infty})=0$. Since the zero set of $X$ is contained in $E$ by Proposition \ref{sexxy}, we must have that $x_{\infty}\in E$. Let $x_{i}:=\gamma_{x}(t_{i})$. Then plugging $t_{i}$ into \eqref{star} yields the fact that
$$|h(x_{i})|=|h(x)|e^{-\lambda t_{i}}\to_{i\to\infty}\left\{
\begin{array}{rl}
+\infty & \text{if $\lambda < 0$},\\
0 & \text{if $\lambda > 0$.}
\end{array} \right.$$
Since $x_{i}\to x_{\infty}\in E$ as $i\to\infty$, we conclude from the maximum principle that $\lambda>0$ as required.
\end{proof}

We have just established that $JX$ is an ``algebraic'' Reeb vector field on $C_{0}$ in the sense of \cite[Definition 2.4]{collinss}.
By \cite[Proposition 2.7]{collinss}, this coincides with the definition of Reeb field in the sense of \cite[Definition 2.6]{frank}. In particular, 
by \cite[Lemma 2.2]{frank} there exists a $T$-invariant K\"ahler cone metric $g_{0}$ with radial function $r$ such that $r\partial_{r}=X$.
\end{enumerate}
\end{proof}

\subsection{Set-up of the complex Monge-Amp\`ere equation}

Recall the $T$-invariant K\"ahler cone metric $g_{0}$ on $C_{0}$ from Proposition \ref{jenny}(iv). This encodes the model geometry at infinity. Let $\omega_{0}$ 
denote the K\"ahler form of $g_{0}$. We now construct a background metric $g$ on $M$ that is asymptotic to $g_{0}$.
\begin{prop}\label{mainprop}
\begin{enumerate}
\item  There exists a complete $T$-invariant K\"ahler metric $g$ on $M$ with K\"ahler form $\omega$ such that outside a compact subset of $M$ containing $E$,
$\omega=\pi^{*}(\omega_{0}+\rho_{\omega_{0}})$. In particular, $\pi_{*}\omega-\omega_{0}=O(r^{-2})$ with $g_{0}$-derivatives.
In addition, there exists a smooth real-valued torus-invariant function $F=c_{0}-\frac{s_{\omega_{0}}}{2}+O(r^{-4})$ with $g_{0}$-derivatives such that
\begin{equation}\label{ivin}
i\partial\bar{\partial}F=\rho_{\omega}+\frac{1}{2}\mathcal{L}_{X}\omega-\omega.
\end{equation}
Here, $c_{0}\in\mathbb{R}$, $\rho_{\omega_{0}}$ (respectively $\rho_{\omega}$) denotes the Ricci form of $\omega_{0}$ (resp.~$\omega$), and $s_{\omega_{0}}$ denotes the scalar curvature of $\omega_{0}$.

\item There exists a unique $T$-invariant real-valued function $f\in C^{\infty}(M)$ with
$-\omega\lrcorner JX=df$ such that outside a compact subset of $M$ containing $E$,
$f=\pi^{*}\left(\frac{r^{2}}{2}-n\right)$ and
\begin{equation}\label{normal12}
\Delta_{\omega}f+f-\frac{X}{2}\cdot f=-\frac{X}{2}\cdot F.
\end{equation}
In particular,
\begin{equation}\label{normal2}
\Delta_{\omega}f+f-\frac{X}{2}\cdot f=O(r^{-2})\quad\textrm{with $g_{0}$-derivatives,}
\end{equation}
and $f\to+\infty$ as $r\to+\infty$, hence is proper. Moreover,
\begin{equation}\label{bds-cov-der-f}
|X\cdot f-2f|\leq C\qquad\textrm{and}\qquad|\nabla^{g,\,k}(\nabla^{g,\,2}f-g)|_{g}\leq \frac{C_{k}}{(f+C)^{1+\frac{k}{2}}}\qquad\textrm{for all $k\geq 0$}.
\end{equation}
  \item Recall that $\tilde{\omega}$ denotes the K\"ahler form of the complete shrinking gradient K\"ahler-Ricci soliton $(M,\,\tilde{g},\,X)$. 
  There exists a $T$-invariant smooth function $\varphi\in C^{\infty}(M)$ such that  
  $\tilde{\omega}=\omega+i\partial\bar{\partial}\varphi>0$ satisfying the complex Monge-Amp\`ere equation
\begin{equation}\label{cmaa}
(\omega+i\partial\bar{\partial}\varphi)^{n}=e^{F+\frac{X}{2}\cdot\varphi-\varphi}\omega^{n},
\end{equation}
where $F$ is as in part (i). 
\end{enumerate}
\end{prop}

Note that a priori nothing is known about the growth of $\varphi$ at infinity.

\begin{proof}[Proof of Proposition \ref{mainprop}]
This proposition has already been noted in the literature, namely in \cite[Proposition 3.2]{babu}, where the relevant references are given. However, 
for the convenience of the reader, we provide the details here.
\begin{enumerate}
\item When $E=\emptyset$ so that $M$ is biholomorphic to $\mathbb{C}^{n}$ by Proposition \ref{sexxy}(i), \cite[Lemma 4.5]{cif-esp} gives us a global $T$-invariant K\"ahler metric $g$ on $\mathbb{C}^{n}$ which is equal to the $T$-invariant K\"ahler cone metric $g_{0}$ of Proposition \ref{jenny} outside a compact subset containing the origin. One just amalgamates the flat metric in a neighbourhood of the origin with the cone metric at infinity in a standard way and then averages the resulting metric over the action of $T$.

When $E\neq\emptyset$, working with the shrinking soliton $\tilde{\omega}$ and the strictly positive curvature form $i\Theta$ given by Claim \ref{tosatti}, we proceed as in \cite[Proposition 3.1]{con-der}. In what follows, we identify $M\setminus E$ and the complement of the apex of $C_{0}$ via $\pi$.

Since $\rho_{\omega_{0}}$ is the curvature form of the hermitian metric on $-K_{C_{0}}$ induced by $\omega_{0}$, there exists a smooth function $u$ on $M\setminus E$ such that $$i\Theta=\rho_{\omega_{0}}-i\partial\bar{\partial}u.$$ We proceed as in the proof of \cite[Lemma 2.15]{Conlon}. Let $\alpha>0$ and let $\psi_{\alpha}:\R^+\to\R^+$ be smooth with $\psi_{\alpha}',\psi_{\alpha}''\geq 0$ and
$$\psi_{\alpha}(t) = \begin{cases}
\left(\frac{\epsilon}{3}\right)^{2\alpha} & \textrm{if}\;\,t<\left(\frac{\epsilon}{2}\right)^{2\alpha},\\
t & \textrm{if}\;\,t>\epsilon^{2\alpha}.
\end{cases}
$$
Then $\Psi_{\alpha}:=\psi_{\alpha}\circ r^{2\alpha}: M \to \R^+$ satisfies
$$
i\partial\bar{\partial}\Psi_{\alpha} =
\begin{cases}
0 &\textrm{on}\;E \cup \{0 < r <\frac{\epsilon}{2}\},\\
\psi_{\alpha}'' i\p r^{2\alpha}\wedge\bar{\p}r^{2\alpha}+\psi_{\alpha}'i\p\bar{\p}r^{2\alpha} &\textrm{on}\;\{r > \frac{\epsilon}{4}\}.
\end{cases}
$$
Clearly $i\partial\bar{\partial}\Psi_{\alpha}\geq 0$ everywhere on $M$ and $i\partial\bar{\partial}\Psi_{\alpha}=i\partial\bar{\partial}r^{2\alpha}>0$ on $\{r>\epsilon\}$. {Also}, fix a cutoff function $\zeta:\mathbb{R}^{+}\to\R^{+}$ with
$$
\zeta(t) = \begin{cases}
0 & \textrm{if}\;\,t<2,\\
1 & \textrm{if}\;\,t>3,
\end{cases}
$$
and for $R>4\epsilon$, define $\zeta_{R}:M\to\mathbb{R}$ by $\zeta_{R}:=\zeta\circ(r/R)$. We construct
$$\hat{\omega}:=i\Theta+i\p\bar{\p}(\zeta_{\epsilon}u) + Ci\p\bar{\p}(\zeta_{R}\Psi_{\alpha})+i\p\bar{\p}\Psi_{1}$$
with $C$ and $R$ to be determined and with $\alpha\in(0,\,1)$ fixed. {Note that} $$\textrm{$\hat{\omega} = i\Theta+Ci\p\bar{\p}(\zeta_{R}\Psi_{\alpha})+i\p\bar{\p}\Psi_{1}\geq i\Theta+i\p\bar{\p}\varphi>0$ on $E\cup\{0< r < 2\epsilon\}$}$$ because $\Psi_{\alpha}$ and $\Psi_{1}$ are plurisubharmonic; $\hat{\omega} =\rho_{\omega_{0}} + i\p\bar{\p}r^{2} > 0$ on $\{2R < r \}$, after increasing $R$ if necessary, because $|\rho_{\omega_{0}}|_{i\partial\bar{\partial}r^{2}} = O(r^{-2})$; $\hat{\omega}>0$ on $\{2\epsilon\leq r\leq 3\epsilon\}$ by compactness if $C$ is made large enough; $\hat{\omega}=\rho_{\omega_{0}}+Ci\partial\bar{\partial}r^{2\alpha}+i\partial\bar{\partial}r^{2}>0$ on $\{3\epsilon<r<R\}$ after further increasing $C$ independently of $R$, which one can do since $|\rho_{\omega_{0}}|_{i\partial\bar{\partial}r^{2\alpha}}=O(r^{-2\alpha})$; and finally, $\hat{\omega} > 0$ on $\{2R \leq r \leq 3R\}$ after further increasing $R$ if necessary, since $\Psi_{\alpha}$ is of lower order compared to $\Psi_{1}$. In conclusion, $\hat{\omega}_{c}$ is a genuine K\"ahler form {on $M$} for suitable choices of $C$ and $R$ with $\hat{\omega} = \omega_{0}+\rho_{\omega_{0}}$ on $\{r>2R\}$.

We next average $\hat{\omega}$ over the action of the torus $T$ on $M$ induced by the flow of the vector field $J_{0}r\partial_{r}$ on $C_{0}$ by setting
$$\omega:=\frac{1}{|T|}\int_{T}\psi_{g}^{*}\hat{\omega}\,d\mu(g)=i\widetilde{\Theta}+i\partial\bar{\partial}\tilde{u},$$
where $\psi_{g}:M\to M$ is the automorphism of $M$ induced by $g\in T$, $\widetilde{\Theta}$ is the average of $\Theta$ over the action of $T$, and where $\tilde{u}$ is defined implicitly. Since there is a path in $T$ connecting $g$ to the identity, we have that $\psi_{g}^{*}[\hat{\omega}]=[\psi_{g}^{*}\hat{\omega}]=[\hat{\omega}]$, from which it follows that $[\omega]=[\hat{\omega}]$. Moreover, it is clear $\omega$ is $T$-invariant. Finally, since $T$ acts by holomorphic isometries on the slices of the cone $C_{0}$, we have that $\psi_{g}^{*}\rho_{\omega_{0}}=\rho_{\omega_{0}}$ and $\psi_{g}^{*}\omega_{0}=\omega_{0}$ for every $g\in T$. Hence $\omega=\omega_{0}+\rho_{\omega_{0}}$ on $\{r>2R\}$ also.
This gives us the desired metric in this case.

Next, we derive \eqref{ivin}. Since $\mathcal{L}_{JX}\omega=0$ and because $M$ is simply connected by Proposition \ref{jenny}(ii) and $JX$ is Killing,
we can find a smooth real-valued function $\theta_{X}$ such that $\omega\lrcorner X=d\theta_{X}\circ J$. This allows us to write
$$\mathcal{L}_{X}\omega=d(\omega\lrcorner X)=i\partial\bar{\partial}\theta_{X}.$$ Since $\mathcal{L}_{Y}\mathcal{L}_{X}\omega=d(\omega\lrcorner[Y,\,X])=0$ for any $Y\in\mathfrak{t}$ because $[X,\,Y]=0$, by averaging over the action of $T$, we may assume that $\theta_{X}$ is invariant under $T$. Moreover, as the difference of two curvature forms, we have that $$\rho_{\omega}-i\widetilde{\Theta}=i\partial\bar{\partial}v$$ for some $v\in C^{\infty}(M)$ . Averaging this equation over the action of the $T$, we may assume that $v$ is invariant under $T$. Next, we write
\begin{equation}\label{sexier}
\begin{split}
\rho_{\omega}+\frac{1}{2}\mathcal{L}_{X}\omega-\omega&=\rho_{\omega}+\frac{1}{2}\mathcal{L}_{X}\omega-i\widetilde{\Theta}-i\partial\bar{\partial}\tilde{u}\\
&=i\partial\bar{\partial}\left(v-\tilde{u}+\frac{1}{2}\theta_{X}\right)\\
&=i\partial\bar{\partial}F
\end{split}
\end{equation}
for $F:=v-\tilde{u}+\frac{1}{2}\theta_{X}\in C^{\infty}(M)$. In particular, notice that $F$ is $T$-invariant.

Next observe that at infinity we have
\begin{equation}\label{sexiest}
\begin{split}
\rho_{\omega}+\frac{1}{2}\mathcal{L}_{X}\omega-\omega&=\rho_{\omega}-(\omega_{0}+\rho_{\omega_{0}})+
\frac{1}{2}\left(\mathcal{L}_{r\p_{r}}\omega_{0}+\underbrace{\mathcal{L}_{r\p_{r}}\rho_{\omega_{0}}}_{=\,0}\right)\\
&=-i\partial\bar{\partial}\log\left(\frac{(\omega_{0}+\rho_{\omega_{0}})^{n}}{\omega^{n}_{0}}\right)
\underbrace{-\omega_{0}+\frac{1}{2}\mathcal{L}_{r\p_{r}}\omega_{0}}_{=\,0}\\
&=-i\partial\bar{\partial}\log\left(\frac{(\omega_{0}+\rho_{\omega_{0}})^{n}}{\omega^{n}_{0}}\right)\\
&=i\partial\bar{\partial}G
\end{split}
\end{equation}
for
\begin{equation*}
\begin{split}
G=G(\omega_{0})&:=-\log\left(\frac{(\omega_{0}+\rho_{\omega_{0}})^{n}}{\omega^{n}_{0}}\right)\\
&=-\log\Bigg(1+\underbrace{n\frac{\omega_{0}^{n-1}\wedge \rho_{\omega_{0}}}{\omega^{n}_{0}}}_{=\,O(r^{-2})}+O(r^{-4})\Bigg)\\
&=-\left(n\frac{\omega_{0}^{n-1}\wedge \rho_{\omega_{0}}}{\omega^{n}_{0}}+O(r^{-4})\right)\\
&=-\frac{s_{\omega_{0}}}{2}+O(r^{-4})\in C^{\infty}_{-2}(M).
\end{split}
\end{equation*}
Notice that $G$ is also $T$-invariant. On subtracting \eqref{sexier} from \eqref{sexiest}, we see that at infinity
\begin{equation}\label{beyonce}
i\partial\bar{\partial}(F-G)=0.
\end{equation}
By Lemma \ref{nice}, we must have that $F-G=c_{0}$ at infinity for some constant $c_{0}$. The desired conclusion follows.

  \item Since $M$ is simply connected by Proposition \ref{jenny}(ii), there exists a smooth real-valued function $f\in C^{\infty}(M)$, defined up to a constant, with $-\omega\lrcorner JX=df$. Any such choice of $f$ is invariant under the action of $T$ by virtue of the fact that $\omega\lrcorner JX$ is invariant under this action. Next, notice that $-\omega_{0}\lrcorner J_{0}r\partial_{r}=d\left(\frac{r^{2}}{2}\right)$, where recall $J_{0}$ is the complex structure on $C_{0}$,
and $\rho_{\omega_{0}}\lrcorner J_{0}r\partial_{r}=0$ because, as is well-known, the Ricci form of a K\"ahler cone metric is a basic $(1,\,1)$-form.
Henceforth suppressing the pullback by $\pi$, we therefore see from the form of $\omega$ given in part (i) that on the complement of a large compact subset $K\subseteq M$ containing $E$,
$$df=-\omega\lrcorner JX=-\left(\omega_{0}+\rho_{\omega_{0}}\right)\lrcorner JX=d\left(\frac{r^{2}}{2}\right),$$
so that $f$ differs from $\frac{r^{2}}{2}$ by a constant on this set, meaning that $f=\frac{r^{2}}{2}+\operatorname{const.}$
on $M\setminus K$. Normalise $f$ so that this constant is equal to $-n$. Then $f=\pi^{*}\left(\frac{r^{2}}{2}-n\right)$ outside a compact set.
What remains to show is that with respect to this normalisation, \eqref{normal2} holds true.

To this end, let $F$ be the smooth function from part (i) satisfying
$$ i \p \bp F = \rho_{\omega} + \frac{1}{2}\mathcal{L}_X \omega-\omega.$$
Using the $JX$-invariance of $F$ and $f$,
contract this equation with $X^{1,\,0}:=\frac{1}{2}(X-iJX)$ and use the Bochner formula to derive that
$$i\bar{\partial}\left(\Delta_{\omega}f-\frac{X}{2}\cdot f+f+\frac{X}{2}\cdot F\right)=0.$$
As a real-valued holomorphic function, we must have that $\Delta_{\omega}f-\frac{X}{2}\cdot f+f+\frac{X}{2}\cdot F$ is
constant on $M$. In light of the fact that on $M\setminus K$,
\begin{equation}\label{bonjour}
\begin{split}
\Delta_{\omega}f-\frac{X}{2}\cdot f+f&=(\Delta_{\omega}-\Delta_{\omega_{0}})f+\Delta_{\omega_{0}}f-\frac{X}{2}\cdot f+f\\
&=O(r^{-2})+\underbrace{\Delta_{\omega_{0}}\left(\frac{r^{2}}{2}-n\right)-\frac{r}{2}\frac{\partial}{\partial r}\left(\frac{r^{2}}{2}-n\right)+\left(\frac{r^{2}}{2}-n\right)}_{=\,0}\\
&=O(r^{-2})
\end{split}
\end{equation}
and $\frac{X}{2}\cdot F=O(r^{-2})=O(f^{-1})$, this constant must be zero so that globally on $M$, so that
\begin{equation*}
\Delta_{\omega}f+f-\frac{X}{2}\cdot f=-\frac{X}{2}\cdot F=O(f^{-1}).
 \end{equation*}
This proves \eqref{normal12} and \eqref{normal2}. Finally, as $f=\pi^{*}\left(\frac{r^{2}}{2}-n\right)$ outside a compact subset of $M$, the first estimate of
\eqref{bds-cov-der-f} is clear. As for the second, part (i) gives us that
$$\left|\nabla^{g,\,k}\left(\Ric(g)+\nabla^{g,2}f-g\right)\right|_{g}=O\left(f^{-2-\frac{k}{2}}\right)\qquad\textrm{for all $k\geq 0$}$$
outside a compact subset of $M$. But then $g$ is asymptotically conical at rate $-2$ by part (i) again, and so $\left|\nabla^{g,\,k}\Ric(g)\right|_{g}=O\left(f^{-1-\frac{k}{2}}\right)$ for all $k\geq 0$.

\item The fact that the shrinking soliton takes the form as stated follows from \eqref{ivin} and the defining equation for a shrinking soliton. Indeed, by definition, any shrinking gradient K\"ahler-Ricci soliton $\tilde{\omega}$ on $M$ with soliton vector field $X$ and soliton potential $\tilde{f}$ satisfies the equation
$$\tilde{\omega}=\rho_{\tilde{\omega}}+i\partial\bar{\partial}\tilde{f},$$
where $\rho_{\tilde{\omega}}$ is the Ricci form of $\tilde{\omega}$. Since $df=-\omega\lrcorner JX$, we know that $X=\nabla^{g}f$, where $g$ is the K\"ahler metric associated to $\omega$, and so from \eqref{ivin} it follows that 
$$\omega=\rho_{\omega}+i\partial\bar{\partial}f-i\partial\bar{\partial}F.$$ Thus,
$$
\tilde{\omega}-\omega=\rho_{\tilde{\omega}}-\rho_{\omega}+i\partial\bar{\partial}\left(\tilde{f}-f+F\right)=i\partial\bar{\partial}\left(\tilde{f}-f+F-\log\left(
\frac{\tilde{\omega}^{n}}{\omega^{n}}\right)\right),
$$
so that $\tilde{\omega}=\omega+i\partial\bar{\partial} \varphi$ with $$\varphi:=\tilde{f}-f+F-\log\left(\frac{\tilde{\omega}^{n}}{\omega^{n}}\right).$$

Next, one derives the complex Monge-Amp\`ere equation as in \cite[Proof of Proposition 3.2]{con-der} using Lemma \ref{nice}. To this end, we write $\omega_{\varphi}:=\omega+i\partial\bar{\partial}\varphi$ for the shrinking soliton $\tilde{\omega}$ as previously explained. Then
\begin{equation}\label{computation}
\begin{split}
0&=\rho_{\omega_{\varphi}}-\omega_{\varphi}+\frac{1}{2}\mathcal{L}_{X}
\omega_{\varphi}\\
&=\rho_{\omega_{\varphi}}-\rho_{\omega}+\rho_{\omega}-\omega_{\varphi}+\frac{1}{2}\mathcal{L}_{X}\omega_{\varphi}\\
&=-i\partial\bar{\p}\log\left(\frac{(\omega+i\partial\bar{\partial}\varphi)^{n}}{\omega^{n}}\right)+
\rho_{\omega}-\omega_{\varphi}+\frac{1}{2}\mathcal{L}_{X}\omega_{\varphi}\\
&=-i\partial\bar{\p}\log\left(\frac{(\omega+i\partial\bar{\partial}\varphi)^{n}}{\omega^{n}}\right)-
i\partial\bar{\partial}\varphi+\frac{1}{2}i\partial\bar{\partial}\left(X\cdot\varphi\right)
+\left(\rho_{\omega}-\omega+\frac{1}{2}\mathcal{L}_{X}\omega\right),\\
\end{split}
\end{equation}
so that
\begin{equation*}
\begin{split}
i\partial\bar{\p}\left(\varphi+\log\left(\frac{(\omega+i\partial\bar{\partial}\varphi)^{n}}{\omega^{n}}\right)-\frac{1}{2}X\cdot\varphi\right)
&=\rho_{\omega}-\omega+\frac{1}{2}\mathcal{L}_{X}\omega.\\
\end{split}
\end{equation*}
Now, recall from part (i) that the right-hand side of this equation is equal to $i\partial\bar{\partial}F$.
Thus, we arrive at the equation
\begin{equation}\label{gorilla}
i\partial\bar{\p}\left(-\varphi+\log\left(\frac{(\omega+i\partial\bar{\partial}\varphi)^{n}}{\omega^{n}}\right)+\frac{1}{2}X\cdot\varphi-F\right)=0.\\
\end{equation}
As the function in the parentheses is invariant under the flow of $JX$,
we can apply Lemma \ref{nice} to determine that 
$$-\varphi+\log\left(\frac{(\omega+i\partial\bar{\partial}\varphi)^{n}}{\omega^{n}}\right)+\frac{1}{2}X\cdot\varphi-F=c$$
for some constant $c\in\mathbb{R}$.
By absorbing $c$ into $\varphi$, the result follows.

Finally, the $T$-invariance of $\varphi$ can be seen by averaging over the action of $T$.
\end{enumerate}
\end{proof}

\section{Curvature estimates}

In this section, we use the complex Monge-Amp\`ere equation \eqref{cmaa} derived in the previous section to show that any shrinking gradient K\"ahler-Ricci soliton on a resolution of a K\"ahler cone has quadratic curvature decay. From this, it immediately follows that the soliton is asymptotically conical.

\subsection{Properties of shrinking gradient Ricci solitons}

We begin by recalling the following identities and properties of a shrinking gradient Ricci soliton.
\begin{lemma}\label{id-sol-egs}
Let $(M^{m},\,g,\,X)$ be a complete shrinking gradient Ricci soliton of real dimension $m$ 
satisfying $\Ric(g)+\frac{1}{2}\mathcal{L}_{X}g=g$ with scalar curvature $R_{g}$ and soliton vector field $X=\nabla^{g}f$ for a smooth real-valued function $f:M\to\mathbb{R}$.
Then:
\begin{align}
& R_g+\Delta_g f  =m>0, \label{equ:1} \\
&\nabla^g R_g= 2\Ric(g)(X,\,\cdot)^{\sharp}, \label{equ:2} \\
&d\left(|X|_g^2+R_g-2f\right)=0. \label{equ:3}
\end{align}
In addition, $R_{g}$ satisfies the elliptic equation
\begin{equation}\label{ell-equ-scal-curv}
\frac{1}{2}\Delta_f R_g= R_g-|\Ric(g)|^2_g,
\end{equation}
where $\Delta_{g,\,X}:=\Delta_{g}-X\cdot$.
\end{lemma}
The proof of the above identities can be found for instance in \cite[Chapter 4]{Cho-Lu-Ni-Boo}. 

\noindent\fbox{%
    \parbox{\textwidth}{%
We henceforth normalize the potential function $f$ of a shrinking gradient Ricci soliton $(M,\,g,\,X)$ with $X=\nabla^{g}f$
so that
\begin{equation}\label{normal-f}
\Delta_{g,\,X}f+2f=0.
\end{equation}}}

We also have:
\begin{theorem}\label{prop-sum-sgrs}
Let $(M^{m},\,g,\,X)$ be a complete non-compact shrinking gradient Ricci soliton of real dimension $m$ 
satisfying $\Ric(g)+\frac{1}{2}\mathcal{L}_{X}g=g$ with scalar curvature $R_{g}$ and soliton vector field $X=\nabla^{g}f$ for a smooth real-valued function $f:M\to\mathbb{R}$.
Then the following hold true.
\begin{enumerate}
\item \textnormal{(Non-negativity of the scalar curvature {\cite[Theorem 1.3(ii)]{zhang12}}).}~$R_{g}\geq0$. \\

\item \textnormal{(Growth of the soliton potential {\cite[Theorem 1.1]{caoo}}).}~For $x\in M$, $f$ satisfies the estimates
$$\frac{1}{4}(d_{g}(p,\,x)-c_{1})^{2}-C\leq f(x)\leq\frac{1}{4}(d_{g}(p,\,x)+c_{2})^{2}$$
for some $C>0$, where $d_{g}(p,\,\cdot)$ denotes the distance to a fixed point $p\in M$ with respect to $g$.
Here, $c_{1}$ and $c_{2}$ are positive constants depending only on the real dimension of $M$ and
the geometry of $g$ on the unit ball $B_{p}(1)$ based at $p$.\\

\item \textnormal{(Polynomial volume growth \cite[Theorem 1.2]{caoo}).}~For each $x\in M$, there exists a positive constant $C>0$
such that $$\operatorname{vol}_{g}(B_{r}(x))\leq Cr^{m}\quad\textrm{for $r > 0$ sufficiently large}.$$\\
\end{enumerate}
\end{theorem}

%\begin{proof}
%Since $(M,g,\nabla^gf)$ is a complete shrinking gradient Ricci soliton, the result of \cite{zhang12} implies that the scalar curvature of $g$ is nonnegative, no matter if the curvature is unbounded or not. 

%The lower bound on $f$ is due to \cite{Cao-Zhou}. The upper bound follows from the inequality $|\nabla^gf|^2_g\leq 2f$ on $M$ obtained by combining the nonnegativity of $\RR_g$ together with the soliton identity $|\nabla^gf|^2_g+\RR_g=2f$. \textcolor{red}{Should there be $2f+m$ here?}

%Finally, the volume growth of geodesic balls is proved in \cite[Theorem $1.4$]{Mun-Wan-Adv}.

%Finally, the proof of the growth on Ricci curvature can be found in \cite[Theorem $1.1$]{Mun-Ses}.
%\end{proof}

\subsection{Initial rough extrinsic bounds}
Recall the setup and notation of Proposition \ref{mainprop} from Section \ref{setupp}. By part (iii) of that proposition, any complete shrinking gradient K\"ahler-Ricci soliton $\tilde{\omega}$ on $M$ is of the form $\omega_{\varphi}:=\omega+i\partial\bar{\partial}\varphi>0$, where $\varphi:M\to\mathbb{R}$ is a smooth $T$-invariant real-valued function satisfying
\eqref{cmaa}. We write $g_{\varphi}$ for the corresponding K\"ahler metric. In light of Proposition \ref{mainprop}(ii), we can write $X=\nabla^{g_{\varphi}}f_{\varphi}$, 
where $f_{\varphi}:M\to\mathbb{R}$ is a smooth real-valued function defined by $f_{\varphi}:=f+\frac{X}{2}\cdot\varphi$. With this, we have the following bounds on $X\cdot\varphi$.
\begin{prop}\label{prop-0-t-1}
Let $\varphi$ satisfy \eqref{cmaa}. Then there exists $C>0$ such that 
$$-2f-C\leq X\cdot\varphi\leq 2f_{\varphi}+C.$$
\end{prop}

\begin{proof}
From Theorem \ref{prop-sum-sgrs}, we know that the soliton potential $f_{\varphi}$ of $(M,\,\omega_{\varphi},\,X)$ is proper and bounded from below. Because 
of this and the fact that $f_{\varphi}$ is continuous, $f_{\varphi}$ must attain a minimum. The fact that $X=\nabla^{g_{\varphi}}f_{\varphi}$ means that this minimum lies at a point $p\in M$ in the zero set of $X$. We therefore see that for all $x\in M$, $f_{\varphi}(x)\geq f_{\varphi}(p)=f(p)\geq -C$, where $C>0$ is a positive constant.
Here, the last inequality follows from the fact that $f$ is bounded from below, as seen from Proposition \ref{mainprop}(ii). This yields the asserted lower bound on $X\cdot\varphi$.
The lower bound on $f$ also gives us the asserted upper bound on $X\cdot\varphi$ in the following way: $X\cdot\varphi=2f_{\varphi}-2f\leq 2f_{\varphi}+2C$.
\end{proof}

We next have the following estimate on $X\cdot X\cdot\varphi$.
\begin{prop}\label{prop-1-t-1}
Let $\varphi$ satisfy \eqref{cmaa}. Then $X\cdot X\cdot \varphi\leq 2X\cdot\varphi$ outside a compact subset of $M$. In particular, there exists a positive constant $C>0$ such that $X\cdot\varphi-2\varphi\leq C$ on $M$, and there exist positive constants $C,\,C'>0$ such that $X\cdot\varphi\leq Cf+C'$ and $\varphi\leq Cf+C'$ on $M$.

Finally, there exist positive constants $C,\,C'>0$ such that $f_{\varphi}\leq Cf+C'$ on $M$.
\end{prop}

\begin{proof}
The normalisation \eqref{normal-f} yields the identity $\Delta_{g_{\varphi}}f_{\varphi}-X\cdot f_{\varphi}+2f_{\varphi}=0$. 
This identity, together with \eqref{equ:1} and Theorem \ref{prop-sum-sgrs}(i), leads us to the inequality
\begin{equation*}
\begin{split}
0\leq R_{g_{\varphi}}&=2n-\Delta_{g_{\varphi}}f_{\varphi}\\
&=2n-X\cdot f_{\varphi}+2f_{\varphi}\\
&=\left(2n+2f-X\cdot f\right)-\frac{X}{2}\cdot X\cdot\varphi+X\cdot\varphi\\
%+2f_{\varphi}-X\cdot f_{\varphi}=(2n+2f-X\cdot f)+X\cdot \varphi-\frac{X}{2}\cdot X\cdot \varphi.
\end{split}
\end{equation*}
Next, by Proposition \ref{mainprop}(ii), it is clear that $2n+2f-X\cdot f$ vanishes identically outside a compact subset of $M$.
Thus, there exists $C>0$ such that $X\cdot X\cdot \varphi\leq 2X\cdot \varphi$ on the set $\{f\geq C\}$. This gives the first claim of the proposition.
Now, this differential inequality can be viewed as a first order linear ordinary differential inequality in $X\cdot \varphi$. By integrating along the flow lines of $X$, one sees that $X\cdot \varphi\leq Cf$ for $f$ large enough. By integrating once again along the flow lines of $X$, we further find that $\varphi\leq Cf$ for $f$ large enough. This yields the two claimed consequences.

Finally, the last statement is a straightforward consequence of the fact that $f_{\varphi}=f+\frac{X}{2}\cdot\varphi$, combined with the previously attained upper bound on $X\cdot \varphi$.
\end{proof}

Combining Propositions \ref{prop-0-t-1} and \ref{prop-1-t-1}, we arrive at:
\begin{corollary}\label{coro-dom-varphi-X-der}
Let $\varphi$ satisfy \eqref{cmaa}. Then there exist positive constants $C,\,C'>0$ such that\linebreak $|\varphi|+|X\cdot\varphi|+|f_{\varphi}|\leq Cf+C'$ on $M$. Moreover, there exist positive constants $C,\,C'>0$ such that $-C\leq f_{\varphi}\leq Cf+C'$ on $M$. 
\end{corollary}

\subsection{Bounds on moments} 
We begin with an auxiliary result that will prove useful in this section.
\begin{lemma}\label{lemma-moments}
For the shrinking gradient K\"ahler-Ricci soliton $(M,\,\omega_{\varphi},\,X)$ with $X=\nabla^{g_{\varphi}}f_{\varphi}$ for $f_{\varphi}:=f+\frac{X}{2}\cdot\varphi$, let $u:M\rightarrow\mathbb{R}$ be a nonnegative $C^1_{\operatorname{loc}}$-function and let $\psi:\R\rightarrow[0,\,1]$ be a smooth cut-off function with $\psi\equiv 1$ on $(-\infty,1]$, $\psi\equiv 0$ on $[2,+\infty)$, and with $\psi'\leq 0$. Set $\psi_R(x):=\psi\left(\frac{f_{\varphi}(x)}{R}\right)$ for $R>0$. Then 
\begin{equation*}
\begin{split}
2\int_{M}f_{\varphi}u\psi_R^2\,e^{-f_{\varphi}}\omega_{\varphi}^n&\leq \int_M|X\cdot u|\psi_R^2\,e^{-f_{\varphi}}\omega_{\varphi}^n.
\end{split}
\end{equation*}
 \end{lemma}

\begin{proof}
First note that since $\psi_R$ is compactly supported by Theorem \ref{prop-sum-sgrs}(ii), integration by parts gives us that
\begin{equation}
\begin{split}
-\int_M(X\cdot u)\psi_R^2\,e^{-f_{\varphi}}\omega_{\varphi}^n&=\int_Mg_{\varphi}\left(\psi_R^2\nabla^{g_{\varphi}}(e^{-f_{\varphi}}),\nabla^{g_{\varphi}}u\right)\,\omega_{\varphi}^n \\
&=-\int_M\operatorname{div}_{g_{\varphi}}\left(\psi_R^2\,\nabla^{g_{\varphi}}e^{-f_{\varphi}}\right)u\,\omega_{\varphi}^n\\
&=\int_M(X\cdot\psi_R^2)u\,e^{-f_{\varphi}}\omega_{\varphi}^n-\int_M\psi_R^2u(\Delta_{g_{\varphi}}e^{-f_{\varphi}})\,\omega_{\varphi}^n \\
&=\frac{2}{R}\int_M\underbrace{\psi_R\psi'\left(\frac{f_{\varphi}(x)}{R}\right)(X\cdot f_{\varphi})u}_{\leq\, 0}\,e^{-f_{\varphi}}\omega_{\varphi}^n+\int_M\psi_R^2u\Delta_{g_{\varphi},\,X}f_{\varphi}\,e^{-f_{\varphi}}\omega_{\varphi}^n.
\end{split}
\end{equation}
From \eqref{normal-f}, we therefore find that
\begin{equation*}
\begin{split}
\int_M2 f_{\varphi}u\psi_R^2\,e^{-f_{\varphi}}\omega_{\varphi}^n&\leq \int_M|X\cdot u|\psi_R^2\,e^{-f_{\varphi}}\omega_{\varphi}^n,
\end{split}
\end{equation*}
as claimed. 
\end{proof}

Using this lemma, we present the uniform bounds on the moments of $f$ and other quantities.
\begin{prop}\label{prop-bds-moments}
Let $\varphi$ satisfy \eqref{cmaa}. Then $\int_M|f|^k\,e^{-f_{\varphi}}\omega_{\varphi}^n<\infty$ for all $k\geq 0$.
In particular, $\int_M\left(|\varphi|^k+|X\cdot \varphi|^k\right)\,e^{-f_{\varphi}}\omega_{\varphi}^n<\infty$ and $\int_M\tr_{\omega_{\varphi}}\omega\,e^{-f_{\varphi}}\omega_{\varphi}^n<\infty$
for all $k\geq 0$. 
\end{prop}

\begin{remark}
From the proof of Proposition \ref{prop-bds-moments}, one can also obtain the following more general exponential integrability estimate, provided that there exists $c>0$ such that on $M$, $c\max\{f,0\}\leq \max\{f_{\varphi},0\}$:
$$\int_Me^{\alpha f}\,e^{-f_{\varphi}}\omega_{\varphi}^n<\infty\qquad\textrm{for all $\alpha<c$}.$$
\end{remark}

\begin{proof}
Let $\psi_R$ be as in Lemma \ref{lemma-moments}. We first apply Lemma \ref{lemma-moments} to $u=f^{2k}$ with $k\geq 1$ to obtain for all $R>0$, the bound
\begin{equation*}
\begin{split}
\int_M2f_{\varphi}f^{2k}\psi_R^2\,e^{-f_{\varphi}}\omega_{\varphi}^n&\leq \int_M|X\cdot f^{2k}|\psi_R^2\,e^{-f_{\varphi}}\omega_{\varphi}^n\\
&=2k\int_M|X\cdot f||f|^{2k-1}\psi_R^2\,e^{-f_{\varphi}}\omega_{\varphi}^n\\
&\leq C(k)\int_M\left(f^{2k}+1\right)\psi_R^2\,e^{-f_{\varphi}}\omega_{\varphi}^n,
\end{split}
\end{equation*}
where we use the properties of $f$ given in Proposition \ref{mainprop}(ii) in the last line.
Rearranging, we see that
\begin{equation}\label{lamal}
\begin{split}
\int_M(2f_{\varphi}-C(k))f^{2k}\psi_R^2\,e^{-f_{\varphi}}\omega_{\varphi}^n\leq C(k)\int_M\psi_R^2\,e^{-f_{\varphi}}\omega_{\varphi}^n.
\end{split}
\end{equation}
Thus, it follows that
\begin{equation*}
\begin{split}
C(k)\int_{\{f_{\varphi}\,\geq\,C(k)\}}f^{2k}\psi_R^2\,e^{-f_{\varphi}}\omega_{\varphi}^n&\leq
\int_{\{f_{\varphi}\,\geq\,C(k)\}}(2f_{\varphi}-C(k))f^{2k}\psi_R^2\,e^{-f_{\varphi}}\omega_{\varphi}^n\\
&=
\int_{M}(2f_{\varphi}-C(k))f^{2k}\psi_R^2\,e^{-f_{\varphi}}\omega_{\varphi}^n+\int_{\{f_{\varphi}\,\leq\,C(k)\}}(C(k)-2f_{\varphi})f^{2k}\psi_R^2\,e^{-f_{\varphi}}\omega_{\varphi}^n\\
&\leq C(k)\int_{\{f_{\varphi}\,\leq\,C(k)\}}f^{2k} \,\psi_R^2\,e^{-f_{\varphi}}\omega_{\varphi}^n+C(k)\int_M\psi_R^2\,e^{-f_{\varphi}}\omega_{\varphi}^n\\
&\leq C(k)\int_{\{f_{\varphi}\,\leq\,C(k)\}}f^{2k} \,e^{-f_{\varphi}}\omega_{\varphi}^n+C(k)\int_M\,e^{-f_{\varphi}}\omega_{\varphi}^n,
\end{split}
\end{equation*}
where we have used the lower bound on $f_{\varphi}$ given in Theorem \ref{prop-sum-sgrs}(ii) and \eqref{lamal} in the third line, and the bound $\Psi_{R}\leq1$ in the last line.
We have also increased $C(k)$ where necessary. The desired integrability result now follows from the fact that $f_{\varphi}$ is proper and the fact that
the upper bound is independent of $R>0$. The assertion that $|\varphi|^{k},\,|X\cdot\varphi|^{k}\in L^1(e^{-f_{\varphi}}\omega_{\varphi}^n)$ 
for all $k\geq 0$ is a consequence of our bounds on the moments of $f$ and Corollary \ref{coro-dom-varphi-X-der}.

Finally, regarding the last assertion of the proposition, observe from an integration by parts that
\begin{equation*}
\begin{split}
\int_M(\tr_{\omega_{\varphi}}\omega)\psi_R^2\,e^{-f_{\varphi}}\omega_{\varphi}^n&=\int_M\left(n-\Delta_{\omega_{\varphi}}\varphi\right)\psi_R^2\,e^{-f_{\varphi}}\omega_{\varphi}^n\\
&=\int_M\left(n-\frac{X}{2}\cdot \varphi\right)\psi_R^2\,e^{-f_{\varphi}}\omega_{\varphi}^n-\int_M(\Delta_{\omega_{\varphi},\,X}\varphi)\psi_R^2\,e^{-f_{\varphi}}\omega_{\varphi}^n\\
&=\int_M\left(n-\frac{X}{2}\cdot \varphi\right)\,\psi_R^2\,e^{-f_{\varphi}}\omega_{\varphi}^n+\frac{2}{R}\int_M\psi_R\psi'\left(\frac{f_{\varphi}}{R}\right)\left(\frac{X}{2}\cdot \varphi\right)\,e^{-f_{\varphi}}\omega_{\varphi}^n,
\end{split}
\end{equation*}
which is bounded above independently of $R>1$ because, as previously deduced, $X\cdot\varphi\in L^1(e^{-f_{\varphi}}\omega_{\varphi}^n)$. This shows that $\tr_{\omega_{\varphi}}\omega\in L^1(e^{-f_{\varphi}}\omega_{\varphi}^n)$.
\end{proof}

\subsection{Intrinsic bounds}
We begin with an $L^{2}$-maximum principle that will prove useful in this section.
\begin{lemma}\label{l2-ppe-max}
For the shrinking gradient K\"ahler-Ricci soliton $(M,\,g_{\varphi},\,X)$ with volume form $d\mu_{g_{\varphi}}$ and with $X=\nabla^{g_{\varphi}}f_{\varphi}$ for $f_{\varphi}:=f+\frac{X}{2}\cdot\varphi$, let $G:\R\rightarrow \R$ be a $C^1$-function such that $\int_M e^{-G(f_{\varphi})}d\mu_{g_{\varphi}}<\infty$, and let $U$ be a smooth function on $\{f_{\varphi}\geq R\}$ such that
\begin{equation}\label{inequ-l2-ppe-max}
\Delta_{g_{\varphi}}U-\nabla^{g_{\varphi}}_{\nabla^{g_{\varphi}}(G(f_{\varphi}))}U\geq V\cdot U
\end{equation}
for some positive potential $V$. If $U\in L^2(e^{-G(f_{\varphi})}d\mu_{g_{\varphi}})$, then $U$ is bounded from above, i.e., there exists a positive constant $C>0$ such that $U\leq C$ on $\{f_{\varphi}\geq R\}$ for some $R>0$.
\end{lemma}

\begin{remark}
In practice, Lemma \ref{l2-ppe-max} will be applied to the function $G(x):=x+\beta_1 \log x+\beta_2x^{-1}$ for $x>0$, with coefficients $\beta_1,\,\beta_2\in\mathbb{R}$.
\end{remark}

\begin{proof}[Proof of Lemma \ref{l2-ppe-max}]
First observe that for any nonnegative constant $C$, the function $U-C$ satisfies \eqref{inequ-l2-ppe-max} since the potential $V$ is positive.
Choose $R>0$ sufficiently large so that the level set $\{f_{\varphi}= R\}$ is a smooth compact hypersurface.
This is possible because the zero set of $X$ is compact \cite[Proposition 3.5]{JunshengSong} and $f_{\varphi}$ is proper by Theorem \ref{prop-sum-sgrs}(ii). 
Next, choose a constant $C$ such that $\max_{\{f_{\varphi}= R\}}U\leq C$ and define the function $u:=(U-C)_{+}=\max\{U-C,\,0\}$. This function is locally Lipschitz, zero on $\{f_{\varphi}= R\}$, 
everywhere nonnegative, and lies in $L^2(e^{-G(f_{\varphi})}d\mu_{g_{\varphi}})$ by virtue of the fact that $\int_Me^{-G(f_{\varphi})}\,d\mu_{g_{\varphi}}<\infty$ by assumption. An integration by parts on $\{f_{\varphi}\geq R\}$ then gives us that
\begin{equation*}
\begin{split}
0&=\int_{\{f_{\varphi}\,=\, R\}}u\nabla_{\textbf{n}}u\,e^{-G(f_{\varphi})}d\sigma_{g_{\varphi}}=\int_{\{f_{\varphi}\geq R\}}\operatorname{div}_{g_{\varphi}}\left(ue^{-G(f_{\varphi})}\nabla^{g_{\varphi}}u\right)\,d\mu_{g_{\varphi}}\\
&=\int_{\{f_{\varphi}\geq R\}}\left(|\nabla^{g_\varphi} u|^2_{g_\varphi} + u \left( \Delta_{g_\varphi} u - \nabla^{g_\varphi}_{\nabla^{g_\varphi} G(f_\varphi)} u \right) \right) \,e^{-G(f_\varphi)}d\mu_{g_\varphi}\\
&\geq\int_{\{f_{\varphi}\,\geq\, R\}}\left(|\nabla^{g_{\varphi}}u|^2_{g_{\varphi}}+Vu^2\right)\,e^{-G(f_{\varphi})}d\mu_{g_{\varphi}}\geq 0,
\end{split}
\end{equation*}
where $\textbf{n}$ denotes the unit outward pointing normal to the level set $\{f_{\varphi}= R\}$ and $d\sigma_{g_{\varphi}}$ denotes the $(2n-1)$-Hausdorff measure induced by the Riemannian metric $g_{\varphi}$ on $\{f_{\varphi}= R\}$. From this, the conclusion follows.
\end{proof}

We use the previous lemma to prove the following lower bound on $\varphi$.
\begin{prop}[{Lower intrinsic bound on $\varphi$}]\label{prop-low-bd-phi}
Let $\varphi$ satisfy \eqref{cmaa}. Then there exists a positive constant $C>0$ such that $\varphi\geq -C(f_{\varphi}+1)$ on $M$.
\end{prop}

\begin{proof}
Using the fact that $\log(1+x)\leq x$ for all $x>-1$, first observe that $\Delta_{g_{\varphi},\,X}\varphi+2\varphi\leq F\leq2C$ 
for some constant $C>0$ by \eqref{cmaa}. Letting $\Phi:=\varphi-C$, we therefore have that $\Delta_{g_{\varphi},\,X}\Phi+2\Phi\leq0$.
Next, from our normalisation \eqref{normal-f}, for some $R\geq 1$ and $\varepsilon\geq0$ we derive that on $\{f_{\varphi}> R\}$:
\begin{equation}
\begin{split}\label{gal-comp-eps}
\Delta_{g_\varphi, X} \left( \frac{\Phi}{f_\varphi^{1+\varepsilon}} \right) &= \frac{\Delta_{g_\varphi, X}\Phi}{f_\varphi^{1+\varepsilon}} + 2g_\varphi(\nabla^{g_\varphi}\Phi, \nabla^{g_\varphi} f_\varphi^{-1-\varepsilon}) + 2(1+\varepsilon)\frac{\Phi}{f_\varphi^{1+\varepsilon}} + (1+\varepsilon)(2+\varepsilon)\frac{X \cdot f_\varphi}{f_\varphi^2} \cdot \frac{\Phi}{f_\varphi^{1+\varepsilon}} \\
&\leq -2\left(\frac{\Phi}{f_\varphi^{1+\varepsilon}}\right) + 2g_\varphi(\nabla^{g_\varphi}\Phi, \nabla^{g_\varphi} f_\varphi^{-1-\varepsilon}) + 2(1+\varepsilon)\frac{\Phi}{f_\varphi^{1+\varepsilon}} + (1+\varepsilon)(2+\varepsilon)\frac{X \cdot f_\varphi}{f_\varphi^2} \cdot \frac{\Phi}{f_\varphi^{1+\varepsilon}} \\
&= 2g_\varphi \left( \nabla^{g_\varphi} \left(\frac{\Phi}{f_\varphi^{1+\varepsilon}}\right), \nabla^{g_\varphi} \log f_\varphi^{-1-\varepsilon} \right) + 2\varepsilon \left(\frac{\Phi}{f_\varphi^{1+\varepsilon}}\right) \\
&\qquad + \left((1+\varepsilon)(2+\varepsilon) - 2(1+\varepsilon)^2\right) \frac{X \cdot f_\varphi}{f_\varphi^2} \cdot \frac{\Phi}{f_\varphi^{1+\varepsilon}} \\
&= 2g_\varphi \left( \nabla^{g_\varphi} \left(\frac{\Phi}{f_\varphi^{1+\varepsilon}}\right), \nabla^{g_\varphi} \log f_\varphi^{-1-\varepsilon} \right) + \varepsilon \left( 2 - (1+\varepsilon)\frac{X \cdot f_\varphi}{f_\varphi^2} \right) \frac{\Phi}{f_\varphi^{1+\varepsilon}}. 
\end{split}
\end{equation}
In particular, $U:=-\frac{\Phi}{f_{\varphi}^{1+\varepsilon}}$ satisfies assumption \eqref{inequ-l2-ppe-max} of Lemma \ref{l2-ppe-max} with $G(f_{\varphi}):=f_{\varphi}+2\log f_{\varphi}^{-1-\varepsilon}$ defined on the set $f_{\varphi}>0$ for $\varepsilon>0$, and $V:=\varepsilon\left(2-(1+\varepsilon)\frac{X\cdot f_{\varphi}}{f_{\varphi}^2}\right)>\varepsilon>0$ 
on $\{f_{\varphi}\geq R\}$ provided that $R\geq R(\varepsilon)$ is chosen large enough. This is possible since
$X\cdot f_{\varphi}\leq 2f_{\varphi}$
by \eqref{equ:3} and Theorem \ref{prop-sum-sgrs}(i). Moreover, $U\in L^2(e^{-G(f_{\varphi})}\omega_{\varphi}^n)$ thanks to Proposition \ref{prop-bds-moments}. Lemma \ref{l2-ppe-max} now tells us that for all $\varepsilon>0$, there exists $C_{\varepsilon}>0$ such that $\Phi\geq-C_{\varepsilon}f_{\varphi}^{1+\varepsilon}$ on $M$.
% $\varphi\geq C-C_{\varepsilon}f_{\varphi}^{1+\varepsilon}\leq-C_{\varepsilon}f_{\varphi}^{1+\varepsilon}$ on $M$. 
Finally, for $\alpha>0$, \eqref{gal-comp-eps} with $\varepsilon=0$ gives us that
\begin{equation}
\begin{split}\label{ppe-min-alpha}
\Delta_{g_{\varphi},\,X}\left(\frac{\Phi}{f_{\varphi}}+\alpha f_{\varphi}\right)&\leq 2g_{\varphi}\left(\nabla^{g_{\varphi}}\left(\frac{\Phi}{f_{\varphi}}+\alpha f_{\varphi}\right),\nabla^{g_{\varphi}}\log f_{\varphi}^{-1}\right)-2\alpha\left(f_{\varphi}-\frac{X\cdot f_{\varphi}}{f_{\varphi}}\right)\\
&<2g_{\varphi}\left(\nabla^{g_{\varphi}}\left(\frac{\Phi}{f_{\varphi}}+\alpha f_{\varphi}\right),\nabla^{g_{\varphi}}\log f_{\varphi}^{-1}\right),
\end{split}
\end{equation}
provided that $f_{\varphi}>R>1$ for some $R$ independent of $\alpha>0$. Since we
already know that $\frac{\Phi}{f_{\varphi}}\geq -Cf_{\varphi}^{\frac{1}{2}}$ outside a compact subset of $M$, the function $\frac{\Phi}{f_{\varphi}}+\alpha f_{\varphi}$ is proper and bounded from below for all $\alpha>0$. The minimum principle applied to \eqref{ppe-min-alpha} at a minimum point of $\frac{\Phi}{f_{\varphi}}+\alpha f_{\varphi}$ 
now gives us that
\begin{equation*}
\min_{\{f_{\varphi}\,>\,R\}}\left\{\frac{\Phi}{f_{\varphi}}+\alpha f_{\varphi}\right\}=\min_{\{f_{\varphi}\,=\,R\}}\left\{\frac{\Phi}{f_{\varphi}}+\alpha f_{\varphi}\right\},
\end{equation*}
where $R$ is assumed large enough so that the (compact) zero set of $X$ is contained in $\{f_{\varphi}<R\}$.
Since the smooth hypersurface $\{f_{\varphi}=R\}$ is independent of $\alpha>0$, we can let $\alpha$ tend to $0$ to obtain the desired result on $\{f_{\varphi}>R\}$:
 $$\frac{\varphi-C}{f_{\varphi}}=\frac{\Phi}{f_{\varphi}}\geq \min_{\{f_{\varphi}=R\}}\frac{\Phi}{f_{\varphi}}>-\infty.$$
From this, the general result follows.
\end{proof}

We next have:
\begin{lemma}\label{lemma-set-eqn-log-vol}
Let $\varphi$ satisfy \eqref{cmaa}. Then outside a sufficiently large compact subset $K\subseteq M$,
\begin{equation*}
\Delta_{\omega_{\varphi},\,X}\left(\frac{X}{2}\cdot \varphi-\varphi\right)=\tr_{\omega_{\varphi}}\rho_{\omega_0},
\end{equation*}
where $\rho_{\omega_{0}}$ denotes the Ricci form of the cone metric $\omega_{0}$.
\end{lemma}

\begin{proof}
From the normalised equation \eqref{normal-f} satisfied by $f_{\varphi}$, we have the following outside a sufficiently large compact set $K\subseteq M$:
\begin{equation*}
\begin{split}
\Delta_{\omega_{\varphi},\,X}\left(\frac{X}{2}\cdot \varphi-\varphi\right)&=\Delta_{\omega_{\varphi},\,X}\left(f_{\varphi}-f-\varphi\right)\\
&=-f_{\varphi}-\Delta_{\omega_{\varphi},\,X}\varphi-\Delta_{\omega_{\varphi},\,X}f\\
&=\frac{X}{2}\cdot f-f-\Delta_{\omega_{\varphi}}\varphi-\Delta_{\omega_{\varphi}}f\\
&=n-\Delta_{\omega_{\varphi}}\varphi-\Delta_{\omega_{\varphi}}f\\
&=\tr_{\omega_{\varphi}}(\omega-i\partial\bar{\partial}f)=\tr_{\omega_{\varphi}}\left(\omega-\frac{1}{2}\mathcal{L}_{X}\omega\right)\\
&=\tr_{\omega_{\varphi}}\rho_{\omega_0}.
\end{split}
\end{equation*}
Here, we have used the asymptotics given by Proposition \ref{mainprop}. This yields the desired identity.
\end{proof}

We use this to next prove a bound on the ratio of the volume forms of the shrinking soliton metric and the background metric. 
\begin{prop}[{Bounds on the ratio of the volume forms}]\label{prop-2-t-1} Let $\varphi$ satisfy \eqref{cmaa}. 
Then there exists a positive constant $C>0$ such that $|X\cdot\varphi-2\varphi|\leq C$ on $M$. In particular, $X\cdot \varphi\geq -Cf_{\varphi}-C'$ and $\varphi\leq Cf_{\varphi}+C'$ for some positive constants $C,\,C'>0$. 
\end{prop}

 \begin{remark}
Proposition \ref{prop-2-t-1} makes use of the quadratic curvature decay of the cone metric in a crucial way. 
 \end{remark}
 
\begin{proof}[Proof of Proposition \ref{prop-2-t-1}]
We begin by proving the upper bound, as that is easier to prove than the lower bound. By Proposition \ref{prop-1-t-1}, we know that $X\cdot X\cdot \varphi\leq 2X\cdot\varphi$ outside a compact subset of $M$. We can rewrite this as $X\cdot (X\cdot\varphi-2\varphi)\leq 0$ outside a compact subset of $M$, which says that the logarithm of the ratio of the volume forms is non-increasing along every integral curve of the soliton vector field $X$. We therefore deduce that $X\cdot\varphi-2\varphi\leq C$ for some positive constant $C>0$ that can be taken as the supremum of $X\cdot\varphi-2\varphi$ over a large compact subset of $M$ of the form $\{f\leq R\}$ for some $R>0$.

As for the lower bound, we proceed by establishing several claims, beginning with:
\begin{claim}\label{claim-1-low-bd}
There exists a positive constant $C>0$ and a compact subset $K\subset M$ such that on $M\setminus K$,
\begin{equation*}
\Delta_{g_{\varphi},\,X}\left(\frac{X}{2}\cdot \varphi-\varphi\right)\leq \frac{2nC}{f}-\frac{2C}{f}(\Delta_{g_{\varphi}}\varphi).
\end{equation*}
\end{claim}

\begin{proof}[Proof of Claim \ref{claim-1-low-bd}]
By Lemma \ref{lemma-set-eqn-log-vol} and the quadratic curvature decay of the background metric $\omega$, one sees that outside a compact subset $K$ of $M$,
\begin{equation*}
\begin{split}
\Delta_{g_{\varphi},\,X}\left(\frac{X}{2}\cdot \varphi-\varphi\right)=2\tr_{\omega_{\varphi}}\rho_{\omega_{0}}&\leq \frac{2C}{f}\tr_{\omega_{\varphi}}\omega=\frac{2nC}{f}-\frac{2C}{f}(\Delta_{g_{\varphi}}\varphi).
\end{split}
\end{equation*}
%where we have used the property of the background metric that $|i\partial\overline{\partial} f-\omega|_{\omega}=O(f^{-1})$.
\end{proof}

We next have:
\begin{claim}\label{claim-2-low-bd}
There exist $R>0$, $C_i>0$, $i=1,2,3,$ and $C'>0$ such that on $\{f_{\varphi}\geq R\}$,
\begin{equation*}
\Delta_{g_{\varphi},\,X}\left(\frac{X}{2}\cdot \varphi-\varphi+C'\left(\frac{\varphi}{f_{\varphi}}\right)\right)\leq \frac{C}{f_{\varphi}}-\left(\frac{C_1}{f}+\frac{C_2}{f_{\varphi}^2}+\left(\frac{C_3}{f_{\varphi}^3}\right)\chi_{\varphi\,\geq\, 0}\right)\left(\frac{X}{2}\cdot \varphi-\varphi+C'\left(\frac{\varphi}{f_{\varphi}}\right)\right).
\end{equation*}
\end{claim}

\begin{proof}[Proof of Claim \ref{claim-2-low-bd}]
We first compute the drift Laplacian of the function $\frac{\varphi}{f_{\varphi}}$ on $M$ in the following way:
\begin{equation}\label{barrier}
\begin{split}
\Delta_{g_{\varphi},\,X}\left(\frac{\varphi}{f_{\varphi}}\right)&=\frac{\Delta_{g_{\varphi},\,X}\varphi}{f_{\varphi}}+2g_{\varphi}\left(\nabla^{g_{\varphi}}\varphi,\nabla^{g_{\varphi}}f_{\varphi}^{-1}\right)+\varphi\Delta_{g_{\varphi},\,X}f_{\varphi}^{-1}\\
&=\frac{\Delta_{g_{\varphi}}\varphi}{f_{\varphi}}-\frac{X\cdot\varphi}{f_{\varphi}}-2\left(\frac{X\cdot\varphi}{f_{\varphi}^2}\right)-
\varphi\left(\frac{\Delta_{g_{\varphi},\,X}f_{\varphi}}{f_{\varphi}^2}\right)+2\varphi\left(\frac{X\cdot f_{\varphi}}{f_{\varphi}^3}\right)\\
&=\frac{\Delta_{g_{\varphi}}\varphi}{f_{\varphi}}-\frac{X\cdot\varphi}{f_{\varphi}}-2\left(\frac{X\cdot\varphi}{f_{\varphi}^2}\right)+2\left(\frac{\varphi}{f_{\varphi}}\right)+2\varphi
\left(\frac{X\cdot f_{\varphi}}{f_{\varphi}^3}\right).
\end{split}
\end{equation}
Here we have used \eqref{normal-f} in the last line. 
Choose $R>0$ sufficiently large so that 
$\{f_{\varphi}\geq R\}\subseteq M\setminus K$, where $K$ is as in 
Claim \ref{claim-1-low-bd}. This is possible thanks to 
Theorem \ref{prop-sum-sgrs}(ii). Next, by 
Proposition \ref{prop-1-t-1},
for fixed $R>0$, we can choose $C'(R)>0$ sufficiently large so that 
$C'f_{\varphi}^{-1}-2Cf^{-1}>0$ on $\{f_{\varphi}\geq R\}$. By Proposition \ref{prop-1-t-1} again and 
Claim \ref{claim-1-low-bd}, it now follows that on $\{f_{\varphi}\geq R\}$:
\begin{equation*}
\begin{split}
\Delta_{g_{\varphi},\,X}\left(\frac{X}{2}\cdot \varphi-\varphi+C'\left(\frac{\varphi}{f_{\varphi}}\right)\right)&\leq \frac{2nC}{f}+\left(\frac{C'}{f_{\varphi}}-\frac{2C}{f}\right)\Delta_{g_{\varphi}}\varphi\\
&\qquad+ C'\left(-\frac{X\cdot\varphi}{f_{\varphi}}-2\left(\frac{X\cdot\varphi}{f_{\varphi}^2}\right)+2\left(\frac{\varphi}{f_{\varphi}}\right)+2\varphi\left(\frac{X\cdot f_{\varphi}}{f_{\varphi}^3}\right)\right)\\
&\leq \frac{2nC}{f}+\left(\frac{C'}{f_{\varphi}}-\frac{2C}{f}\right)\left(F+X\cdot\varphi-2\varphi\right)\\
&\qquad+ C'\left(-\frac{X\cdot\varphi}{f_{\varphi}}-2\left(\frac{X\cdot\varphi}{f_{\varphi}^2}\right)+2\left(\frac{\varphi}{f_{\varphi}}\right)+2\varphi\left(\frac{X\cdot f_{\varphi}}{f_{\varphi}^3}\right)\right)\\
&= \frac{2nC}{f} + \left(\frac{C'}{f_\varphi}-\frac{2C}{f}\right)F-\frac{2C}{f}\left(X \cdot \varphi - 2\varphi\right) \\
&\qquad+\frac{2C'}{f_{\varphi}^2}\left(\varphi\left(\frac{X\cdot f_{\varphi}}{f_{\varphi}}\right)-X\cdot\varphi\right)\\
&\leq \frac{\tilde{C}}{f_{\varphi}}-\frac{2C}{f}(X\cdot\varphi-2\varphi)+\frac{2C'}{f_{\varphi}^2}\left(\varphi\left(\frac{X\cdot f_{\varphi}}{f_{\varphi}}\right)-X\cdot\varphi\right)
\end{split}
\end{equation*}
for some positive constant $\tilde{C}>0$. Here we have used the fact that $\Delta_{g_{\varphi}}\varphi\leq F+X\cdot\varphi-2\varphi$ by \eqref{cmaa} in the second inequality, together with the fact that $f_{\varphi}\leq Cf$ on $\{f_{\varphi}\geq R\}$ by Proposition \ref{prop-1-t-1} in the final inequality. 
By adding and subtracting $\frac{2CC'\varphi}{ff_{\varphi}}$ artificially, one arrives at
\begin{equation*}
\begin{split}
\Delta_{g_{\varphi},\,X}\left(\frac{X}{2}\cdot \varphi-\varphi+C'\left(\frac{\varphi}{f_{\varphi}}\right)\right)&\leq \frac{\tilde{C}}{f_{\varphi}}+\frac{2CC'}{f_{\varphi}}\left(\frac{\varphi}{f}\right)-\frac{2C}{f}\left(X\cdot\varphi-2\varphi+C'\left(\frac{\varphi}{f_{\varphi}}\right)\right)\\
&\qquad+\frac{2C'}{f_{\varphi}^2}\left(\varphi\left(\frac{X\cdot f_{\varphi}}{f_{\varphi}}\right)-X\cdot\varphi\right)\\
&\leq \frac{C''}{f_{\varphi}}-2\left(\frac{C}{f}+\frac{C'}{f_{\varphi}^2}\right)\left(X\cdot\varphi-2\varphi+C'\left(\frac{\varphi}{f_{\varphi}}\right)\right)\\
&\qquad+\frac{2C'}{f_{\varphi}^2}\left(-2\varphi+C'\left(\frac{\varphi}{f_{\varphi}}\right)+\varphi\left(\frac{X\cdot f_{\varphi}}{f_{\varphi}}\right)\right)\\
&= \frac{C''}{f_{\varphi}}-2\left(\frac{C}{f}+\frac{C'}{f_{\varphi}^2}\right)\left(X\cdot\varphi-2\varphi+C'\left(\frac{\varphi}{f_{\varphi}}\right)\right)\\
&\qquad+\frac{2C'\varphi}{f_{\varphi}^3}\left(C'+X\cdot f_{\varphi}-2f_{\varphi}\right)
\end{split}
\end{equation*}
for some positive constant $C''>0$. Here we have used in the second inequality the fact that
$\varphi$ is dominated by $f$ by virtue of Proposition \ref{prop-1-t-1}.

Now, on the one hand, by combining \eqref{bds-cov-der-f} and Proposition \ref{prop-1-t-1}, we see that $X\cdot f_{\varphi}-2f_{\varphi}$ is bounded from above on $M$, in particular on the set where $\varphi\geq0$. On the other hand, on the set where $\varphi\leq 0$, because $\varphi\geq -Cf_{\varphi}$ by Proposition \ref{prop-low-bd-phi}, we know that $\varphi(C'+X\cdot f_{\varphi}-2f_{\varphi})$ is bounded from above by $f_{\varphi}^2$ up to a positive multiplicative constant. Therefore, we find from above that outside a compact set of $M$,
\begin{equation}
\begin{split}\label{ugly-as-f}
\Delta_{g_{\varphi},\,X}\left(\frac{X}{2}\cdot \varphi-\varphi+C'\left(\frac{\varphi}{f_{\varphi}}\right)\right)&\leq \frac{C''}{f_{\varphi}}-2\left(\frac{C}{f}+\frac{C'}{f_{\varphi}^2}\right)\left(X\cdot\varphi-2\varphi+C'\left(
\frac{\varphi}{f_{\varphi}}\right)\right)+\frac{C''\varphi}{f_{\varphi}^3}\chi_{\varphi\,\geq\, 0},
\end{split}
\end{equation}
where $C''>0$ is a positive constant that may vary from line to line. In order to deal the last term of \eqref{ugly-as-f}, observe the following from Proposition \ref{prop-1-t-1}:
\begin{equation*}
\frac{\varphi}{f_{\varphi}^3}\leq C\left(\frac{f}{f_{\varphi}^3}\right)=\frac{C}{f_{\varphi}^2}-\frac{C}{2}\cdot\frac{X\cdot\varphi}{f_{\varphi}^3}=\frac{C}{f_{\varphi}^2}-\frac{C}{f_{\varphi}^3}\left(\frac{X}{2}\cdot \varphi-\varphi+C'\left(\frac{\varphi}{f_{\varphi}}\right)\right)-C\left(\frac{\varphi}{f_{\varphi}^3}\right)\left(1-\frac{C'}{f_{\varphi}}\right),
\end{equation*}
which, after increasing $R$ and absorbing the last term on the right-hand side into the left-hand side if $f_{\varphi}$ is large compared to $C'$, gives on $\{\varphi\,\geq\, 0\}\cap\{f_{\varphi}\geq R\}$:
\begin{equation}\label{ugly-as-f-II}
\frac{\varphi}{f_{\varphi}^3}\leq \frac{C}{f_{\varphi}^2}-\frac{C}{f_{\varphi}^3}\left(\frac{X}{2}\cdot \varphi-\varphi+C'\left(\frac{\varphi}{f_{\varphi}}\right)\right),
\end{equation}
 for a possibly larger constant $C>0$ and the same constant $C'>0$. Combining \eqref{ugly-as-f} and \eqref{ugly-as-f-II} leads to the desired conclusion.
\end{proof}

Now we prove:
\begin{claim}\label{claim-3-low-bd}
With the notation as in Claim \ref{claim-2-low-bd}, there exists $\alpha>0$ and $R>0$ such that on $\{f_{\varphi}\geq R\}$,
\begin{equation*}
\Delta_{g_{\varphi},\,X}\left(\frac{X}{2}\cdot \varphi-\varphi+C'\left(\frac{\varphi}{f_{\varphi}}\right)-\frac{\alpha}{f_{\varphi}}\right)\leq- \left(\frac{C_1}{f}+\frac{C_2}{f_{\varphi}^2}+\left(\frac{C_3}{f_{\varphi}^3}\right)\chi_{\varphi\,\geq\, 0}\right)\left(\frac{X}{2}\cdot \varphi-\varphi+C'\left(\frac{\varphi}{f_{\varphi}}\right)-\frac{\alpha}{f_{\varphi}}\right).
\end{equation*}
\end{claim}

\begin{proof}[Proof of Claim \ref{claim-3-low-bd}]
By \eqref{normal-f} and Claim \ref{claim-2-low-bd}, for $\alpha>0$ we have that 
\begin{equation*}
\begin{split}
\Delta_{g_{\varphi},\,X}\left(\frac{X}{2}\cdot \varphi-\varphi+C'\left(\frac{\varphi}{f_{\varphi}}\right)-\frac{\alpha}{f_{\varphi}}\right)&\leq \frac{C}{f_{\varphi}}-\left(\frac{C_1}{f}+\frac{C_2}{f_{\varphi}^2}+\left(\frac{C_3}{f_{\varphi}^3}\right)\chi_{\varphi\,\geq\, 0}\right)\left(\frac{X}{2}\cdot \varphi-\varphi+C'\left(\frac{\varphi}{f_{\varphi}}\right)\right)\\
&\qquad-\frac{2\alpha}{f_{\varphi}}-2\alpha\left(\frac{|X|^2_{g_{\varphi}}}{f_{\varphi}^3}\right)\\
&\leq - \left(\frac{C_1}{f}+\frac{C_2}{f_{\varphi}^2}+\left(\frac{C_3}{f_{\varphi}^3}\right)\chi_{\varphi\,\geq\, 0}\right)\left(\frac{X}{2}\cdot \varphi-\varphi+C'\left(\frac{\varphi}{f_{\varphi}}\right)\right)\\
&\leq - \left(\frac{C_1}{f}+\frac{C_2}{f_{\varphi}^2}+\left(\frac{C_3}{f_{\varphi}^3}\right)\chi_{\varphi\,\geq\, 0}\right)\left(\frac{X}{2}\cdot \varphi-\varphi+C'\left(\frac{\varphi}{f_{\varphi}}\right)-\frac{\alpha}{f_{\varphi}}\right),
\end{split}
\end{equation*}
where we have chosen $\alpha>0$ so that $2\alpha\geq C$ in the second inequality.
\end{proof}

Finally, we prove:
\begin{claim}\label{claim-4-low-bd}
Let $u$ be a smooth function defined on $\{f_{\varphi}>R\}$ for some $R>0$ such that there exist positive constants $C_i$, $i=1,2,3$, so that:
\begin{equation*}
\Delta_{g_{\varphi},\,X}u\leq  - \left(\frac{C_1}{f}+\frac{C_2}{f_{\varphi}^2}+\left(\frac{C_3}{f_{\varphi}^3}\right)\chi_{\varphi\,\geq\, 0}\right)u.
\end{equation*}
Then there exists $\beta>0$ and $R_{\beta}\geq R$ such that on $\{f_{\varphi}>R_{\beta}\}$,
\begin{equation*}
\Delta_{g_{\varphi},\,X}\left(e^{\beta f_{\varphi}^{-1}}u\right)\leq 2\beta g_{\varphi}\left(\nabla^{g_{\varphi}}f_{\varphi}^{-1},\nabla^{g_{\varphi}}\left(e^{\beta f_{\varphi}^{-1}}u\right)\right)+V\left(e^{\beta f_{\varphi}^{-1}}u\right),
\end{equation*}
where the potential $V$ is defined by:
\begin{equation*}
V:=\frac{2\beta}{f_{\varphi}}+2\beta\left(\frac{|X|^2_{g_{\varphi}}}{f_{\varphi}^3}\right)-\beta^2\left(\frac{|X|^2_{g_{\varphi}}}{f_{\varphi}^4}\right) -\frac{C_1}{f}-\frac{C_2}{f_{\varphi}^2}-\left(\frac{C_3}{f_{\varphi}^3}\right)\chi_{\varphi\,\geq\, 0}>0\qquad\textrm{on $\{f_{\varphi}>R_{\beta}\}$}.
\end{equation*}
\end{claim}

\begin{proof}[Proof of Claim \ref{claim-4-low-bd}]
Using \eqref{normal-f} and our assumption, we bound the drift Laplacian of $e^{\beta f_{\varphi}^{-1}}u$ in the following way:
\begin{equation*}
\begin{split}
\Delta_{g_{\varphi},\,X}\left(e^{\beta f_{\varphi}^{-1}}u\right)&=\left(\Delta_{g_{\varphi},\,X}e^{\beta f_{\varphi}^{-1}}\right)u+2g_{\varphi}\left(\nabla^{g_{\varphi}}e^{\beta f_{\varphi}^{-1}},\nabla^{g_{\varphi}}u\right)+e^{\beta f_{\varphi}^{-1}}\Delta_{g_{\varphi},\,X}u\\
&=\left(\beta \Delta_{g_{\varphi},\,X}f_{\varphi}^{-1}+\beta^2|\nabla^{g_{\varphi}}f_{\varphi}^{-1}|^2_{g_{\varphi}}\right)e^{\beta f_{\varphi}^{-1}}u+2\beta g_{\varphi}\left(\nabla^{g_{\varphi}} f_{\varphi}^{-1},\nabla^{g_{\varphi}}\left(e^{\beta f_{\varphi}^{-1}}u\right)\right)\\
&\quad-2\beta^2|\nabla^{g_{\varphi}}f_{\varphi}^{-1}|^2_{g_{\varphi}}e^{\beta f_{\varphi}^{-1}}u+e^{\beta f_{\varphi}^{-1}}\Delta_{g_{\varphi},\,X}u\\
&=\left(\frac{2\beta}{f_{\varphi}}+2\beta\left(\frac{|X|^2_{g_{\varphi}}}{f_{\varphi}^3}\right)-\beta^2\left(\frac{|X|^2_{g_{\varphi}}}{f_{\varphi}^4}\right)\right)e^{\beta f_{\varphi}^{-1}}u
+2\beta g_{\varphi}\left(\nabla^{g_{\varphi}} f_{\varphi}^{-1},\nabla^{g_{\varphi}}\left(e^{\beta f_{\varphi}^{-1}}u\right)\right)\\
&\qquad
+e^{\beta f_{\varphi}^{-1}}\Delta_{g_{\varphi},\,X}u\\
&\leq\underbrace{\left(\frac{2\beta}{f_{\varphi}}+2\beta\left(\frac{|X|^2_{g_{\varphi}}}{f_{\varphi}^3}\right)-\beta^2\left(\frac{|X|^2_{g_{\varphi}}}{f_{\varphi}^4}\right) -\frac{C_1}{f}-\frac{C_2}{f_{\varphi}^2}-\left(\frac{C_3}{f_{\varphi}^3}\right)\chi_{\varphi\,\geq\, 0}\right)}_{=:V}e^{\beta f_{\varphi}^{-1}}u\\
&\qquad+2\beta g_{\varphi}\left(\nabla^{g_{\varphi}} f_{\varphi}^{-1},\nabla^{g_{\varphi}}\left(e^{\beta f_{\varphi}^{-1}}u\right)\right).
\end{split}
\end{equation*}
This yields the expected result after noting that the potential $V$ is positive on $\{f_{\varphi}>R\}$ once $\beta>0$ is chosen such that $\beta f>C_1 f_{\varphi}$
which is possible by Proposition \ref{prop-1-t-1} and after taking $R_{\beta}\geq R$ sufficiently large so that the remaining term
\begin{equation*}
\frac{\beta}{f_{\varphi}}-\beta^2\left(\frac{|X|^2_{g_{\varphi}}}{f_{\varphi}^4}\right)-\frac{C_2}{f_{\varphi}^2}-\left(\frac{C_3}{f_{\varphi}^3}\right)\chi_{\varphi\,\geq\, 0}>0\qquad\textrm{on $\{f_{\varphi}>R_{\beta}\}$.}
\end{equation*}
\end{proof}

Now, taken together, Claims \ref{claim-3-low-bd} and \ref{claim-4-low-bd}
imply that 
\begin{equation*}
\Delta_{g_{\varphi}}U-\nabla^{g_{\varphi}}_{\nabla^{g_{\varphi}}(G(f_{\varphi}))}U\geq V\cdot U,
\end{equation*}
where $U:=-e^{\beta f_{\varphi}^{-1}}\left(\frac{X}{2}\cdot \varphi-\varphi+C'\frac{\varphi}{f_{\varphi}}-\frac{\alpha}{f_{\varphi}}\right)$, 
$G(f_{\varphi}):=f_{\varphi}+2\beta f_{\varphi}^{-1}$, and $V$ is defined as in Claim \ref{claim-4-low-bd}.
Invoking Lemma \ref{l2-ppe-max} (after appealing to Proposition \ref{prop-bds-moments} to infer that
$U\in L^{2}(e^{-G(f_{\varphi})}d\mu_{g_{\varphi}})$) now yields the result.
\end{proof}

Using Proposition \ref{prop-2-t-1}, we can now prove the following coercive estimate.
\begin{corollary}[{Equivalence of potentials}]\label{coro-pot-bd}
Let $\varphi$ satisfy \eqref{cmaa}. There exists $c>0$ such that $f+\varphi\geq cf$ outside a compact subset of $M$.
In particular, there exists $C>0$ such that $C^{-1}f\leq f_{\varphi}\leq Cf$ outside a compact subset of $M$.
\end{corollary}

\begin{proof}
Observe from Proposition \ref{mainprop}(ii) that outside a compact set of $M$, $X\cdot (f+\varphi)=2(f+n)+X\cdot \varphi\geq 2(f+\varphi) - C$ for some positive constant $C>0$ by virtue of Proposition \ref{prop-2-t-1}. Now, critical points of $f+\varphi$ lie in $\{f_{\varphi}<R\}$ for some $R>0$.
Indeed, the previous computation shows in particular that if $R>0$ is chosen large enough 
so that $X\cdot f=2(f+n)$ on $\{f_{\varphi}\geq R\}$, then
$X\cdot(f+\varphi)=2f_{\varphi}+2n\geq 2R+2n>0$ on this set. Such a choice of $R$ is possible because $f_{\varphi}$ is proper by  
Theorem \ref{prop-sum-sgrs}(ii). By integrating the inequality $X\cdot (f+\varphi)\geq 2R+2n>0$ along the flow lines of $X$, one even sees
that $f+\varphi$ is proper and bounded from below (with logarithmic growth with respect to the distance from a fixed point). In particular, for all $C>0$, there exists $R(C)>0$ such that $\inf_{\{f_{\varphi}=R\}}(f+\varphi)>C$, and so one can integrate the inequality $X\cdot (f+\varphi)\geq 2(f+\varphi) - C$ along the flow lines of $X$ to obtain the required growth on $f+\varphi$.

As for the growth of $f_{\varphi}$, we already know from Proposition \ref{prop-1-t-1} that there exists $C>0$ such that $f_{\varphi}\leq Cf$ outside a compact subset of $M$. The lower bound on $f_{\varphi}$ follows directly from the lower bound on $X\cdot\varphi-2\varphi$ given by Proposition \ref{prop-2-t-1}, together with the previously established lower bound on $f+\varphi$.
\end{proof}

We next prove that the metrics are equivalent.
\begin{prop}[{Equivalence of the metrics}]\label{equiv-met-prop-C^2} Let $\varphi$ satisfy \eqref{cmaa}. 
Then there exists $C>0$ such that on $M$, $$C^{-1}\omega\leq \omega_{\varphi}\leq C\omega.$$
\end{prop}

Before proving this proposition, the next lemma is an appropriate version of Yau's Schwarz lemma \cite{Calabiconj} catered to our purposes. It can be viewed as a parabolic version of it, given the drift term comprising the real holomorphic vector field $X$.

\begin{lemma}[Schwarz lemma]\label{gal-Schwarz-lemma}
Let $s:(M,\omega_{\varphi})\rightarrow (N,\eta)$ be a holomorphic map between two K\"ahler manifolds with associated K\"ahler metrics $g_{\varphi}$ and $h$ respectively, such that $|\partial s|$ never vanishes on $M$. Assume that for some real holomorphic vector field $X$ on $M$, $$\Ric(g_{\varphi})+\frac{1}{2}\mathcal{L}_Xg_{\varphi}\geq g_{\varphi}.$$  Then
\begin{equation*}
\begin{split}
\Delta_{\omega_{\varphi},\,X}\left(\log \tr_{\omega_{\varphi}}(s^*\eta)\right)&\geq 1-\frac{1}{2}\frac{\tr_{\omega_{\varphi}}\left(\mathcal{L}_{X}s^*\eta\right)}{\tr_{\omega_{\varphi}}(s^*\eta)}-\frac{\omega_{\varphi}\otimes_{\tr}\Rm(h)(\partial s,\bar{\partial}s,\partial s,\bar{\partial}s)}{\tr_{\omega_{\varphi}}(s^*\eta)},
\end{split}
\end{equation*}
where
\begin{equation*}
\begin{split}
\omega_{\varphi}\otimes_{\tr}\Rm(h)(\partial s,\bar{\partial}s,\partial s,\bar{\partial}s)&:=g_\varphi^{i\bar{j}}g_\varphi^{k\bar{l}}\Rm(h)_{\alpha\bar{\beta}\gamma\bar{\delta}}\,\partial_i s^{\alpha} \,\overline{\partial_{j} s^{\beta}}\,\partial_ks^{\gamma}\,\overline{\partial_{l} s^{\delta}}
\end{split}
\end{equation*}
and
\begin{equation*}
\begin{split}
(s^*\eta,s^*\eta)_{\omega_{\varphi}}&=(\tr_{\omega_{\varphi}}(s^*\eta))^2-n(n-1)\left(\frac{s^*\eta\wedge s^*\eta\wedge\omega_{\varphi}^{n-2}}{\omega_{\varphi}^{n}}\right).
\end{split}
\end{equation*}
\end{lemma}

\begin{proof}
This is essentially taken from \cite[Theorem 3.2.6]{Bou-Eys-Gue} and \cite[Section 7]{Jef-Maz-Rub}. Here we use the identity
\begin{equation*}
X\cdot  \tr_{\omega_{\varphi}}(s^*\eta)=\tr_{\omega_{\varphi}}\left(\mathcal{L}_{X}s^*\eta\right)-\left(\mathcal{L}_{X}\omega_{\varphi},s^*\eta\right)_{\omega_{\varphi}}.
\end{equation*}
\end{proof}

\begin{proof}[Proof of Proposition \ref{equiv-met-prop-C^2}]
Applying Lemma \ref{gal-Schwarz-lemma} to $(M,\omega_{\varphi})$ and $(N,\eta)=(M,\omega)$ with $s=\Id_M$, we see that
\begin{equation}\label{denali}
\begin{split}
\frac{1}{2}\Delta_{g_{\varphi},\,X}\left(\log\tr_{\omega_{\varphi}}\omega\right)&\geq 1-\frac{\tr_{\omega_{\varphi}}\left(\frac{1}{2}\mathcal{L}_X\omega\right)}{\tr_{\omega_{\varphi}}\omega}-\frac{C}{f}\left(\tr_{\omega_{\varphi}}\omega\right)\\
&\geq -\frac{C}{f}-\frac{C}{f}\left(\tr_{\omega_{\varphi}}\omega\right)
\end{split}
\end{equation}
for some positive constant $C>0$. Using $\frac{\varphi}{f_{\varphi}}$ as a barrier function as in 
the proof of Proposition \ref{prop-2-t-1}, we next derive from \eqref{barrier} and \eqref{denali} that
\begin{equation*}
\begin{split}
\Delta_{g_{\varphi},\,X}\left(\log\tr_{\omega_{\varphi}}\omega-C'\left(\frac{\varphi}{f_{\varphi}}\right)\right)&\geq -\frac{C}{f}-\frac{C}{f}\tr_{\omega_{\varphi}}\omega\\
&\qquad-C'\left(\frac{\Delta_{g_{\varphi}}\varphi}{f_{\varphi}}-\frac{X\cdot\varphi}{f_{\varphi}}
-2\left(\frac{X\cdot\varphi}{f_{\varphi}^2}\right)
+2\left(\frac{\varphi}{f_{\varphi}}\right)+2\varphi\left(\frac{X\cdot f_{\varphi}}{f_{\varphi}^3}\right)\right)\\
&\geq -\frac{C}{f}-\frac{2nC'}{f_{\varphi}}+\left(\frac{2C'}{f_{\varphi}}-\frac{C}{f}\right)\tr_{\omega_{\varphi}}\omega\\
&\qquad+\frac{C'}{f_{\varphi}}\left(X\cdot\varphi-2\varphi+2\left(\frac{X\cdot\varphi}{f_{\varphi}}\right)
-2\varphi\left(\frac{X\cdot f_{\varphi}}{f_{\varphi}^2}\right)\right),
\end{split}
\end{equation*}
where we have used the fact that $\Delta_{g_{\varphi}}\varphi=2n-2\tr_{\omega_{\varphi}}\omega$ in the second inequality.

Now, from the bounds on $X\cdot \varphi$ and $\varphi$ given by Proposition \ref{prop-2-t-1},
we see that there exists a positive constant $C''>0$ such that
\begin{equation*}
\begin{split}
\Delta_{g_{\varphi},\,X}\left(\log\tr_{\omega_{\varphi}}\omega-C'\left(\frac{\varphi}{f_{\varphi}}\right)\right)&\geq-\frac{C}{f}-\frac{C''}{f_{\varphi}}+\left(\frac{2C'}{f_{\varphi}}-\frac{C}{f}\right)\tr_{\omega_{\varphi}}\omega+\frac{C'}{f_{\varphi}}\left(X\cdot\varphi-2\varphi\right)\\
&\geq-\frac{C'''}{f_{\varphi}}+\left(\frac{2C'}{f_{\varphi}}-\frac{C}{f}\right)\tr_{\omega_{\varphi}}\omega,
\end{split}
\end{equation*}
for some positive constant $C'''>0$ that may vary from line to line. Here, we use Proposition \ref{prop-2-t-1} 
and Corollary \ref{coro-pot-bd} to justify the last inequality. Akin to the proof of Claim \ref{claim-3-low-bd}, there exists a large constant $\alpha>0$ such that
\begin{equation*}
\begin{split}
\Delta_{g_{\varphi},\,X}\left(\log\tr_{\omega_{\varphi}}\omega-C'\left(\frac{\varphi}{f_{\varphi}}\right)+\frac{\alpha}{f_{\varphi}}\right)&\geq \left(\frac{2C'}{f_{\varphi}}-\frac{C}{f}\right)\tr_{\omega_{\varphi}}\omega\\
&\geq\left(\frac{C'}{f_{\varphi}}\right)\tr_{\omega_{\varphi}}\omega\\
&\geq\left(\frac{\tilde{C}}{f_{\varphi}}\right)e^{-C'\left(\frac{\varphi}{f_{\varphi}}\right)+\frac{\alpha}{f_{\varphi}}}\tr_{\omega_{\varphi}}\omega\\
&\geq \frac{\tilde{C}}{f_{\varphi}} \left(\log\tr_{\omega_{\varphi}}\omega-C'\frac{\varphi}{f_{\varphi}}+\frac{\alpha}{f_{\varphi}}\right)
\end{split}
\end{equation*}
for some positive constant $\tilde{C}>0$. Here, we choose $C'>0$ sufficiently large so that $C'f\geq Cf_{\varphi}$ in the second line. This we can do by Proposition \ref{prop-1-t-1}.
In the third inequality, we use the fact that $\frac{\varphi}{f_{\varphi}}$ is bounded from below by Proposition \ref{prop-low-bd-phi}, and in the last line 
we use the elementary non-sharp inequality $\log x\leq x$ for $x>0$. 

Now, observe that $\log \tr_{\omega_{\varphi}}\omega\in L^2(e^{-f_{\varphi}}\omega_{\varphi}^n)$. Indeed, on the one hand, the arithmetic-geometric mean inequality together with \eqref{cmaa}
gives us that
\begin{equation*}
\log \tr_{\omega_{\varphi}}\omega\geq \log n -\frac{1}{n}\log\left(\frac{\omega_{\varphi}^n}{\omega^n}\right)=\log n-\frac{1}{2n}\left(X\cdot \varphi-2\varphi+2F\right)\geq -C
\end{equation*}
for some positive constant $C>0$, thanks to Proposition \ref{prop-2-t-1}. On the other hand, since there exists $C>0$ such that $\log x\leq \sqrt{x}+C$ for $x>0$, 
we see that $\log \tr_{\omega_{\varphi}}\omega\leq \sqrt{\tr_{\omega_{\varphi}}\omega}+C$. After applying Young's inequality, we therefore deduce that there exists a positive constant $C>0$ such that on $M$, $|\log \tr_{\omega_{\varphi}}\omega|^2\leq C\left(\tr_{\omega_{\varphi}}\omega+1\right)$. Proposition \ref{prop-bds-moments} now ensures that $\log \tr_{\omega_{\varphi}}\omega\in L^2(e^{-f_{\varphi}}\omega_{\varphi}^n).$

To conclude, we apply Lemma \ref{l2-ppe-max} to $U:=\log\tr_{\omega_{\varphi}}\omega-C'\left(\frac{\varphi}{f_{\varphi}}\right)+\frac{\alpha}{f_{\varphi}}$, $G(f_{\varphi}):=f_{\varphi}$, and $V:=\frac{\tilde{C}}{f_{\varphi}}$, which leads to the boundedness of $U$. This in turn gives the boundedness of the function $\log\tr_{\omega_{\varphi}}\omega$ thanks to Proposition \ref{prop-2-t-1}.
\end{proof}

The next proposition establishes an a priori $C^3$-local bound on the function $\varphi$ satisfying \eqref{cmaa}. The method of proof follows along the lines of \cite[Proposition 6.9]{con-der}, which itself is based on \cite{Pho-Ses-Stu}.

\begin{prop}[{$C^3$-bounds on the difference $\omega_{\varphi}-\omega$}]\label{prop-bds-C3}
Let $\varphi$ satisfy \eqref{cmaa}.
Then there exists $C>0$ such that $|\nabla^{g} (\omega_{\varphi}-\omega)|_g\leq C,$
where $g$ is the K\"ahler metric corresponding to the background K\"ahler form 
$\omega$ from Proposition \ref{mainprop}(i).
\end{prop}

Before proving Proposition \ref{prop-bds-C3}, we need a lemma dictating the evolution equation of the norm squared of the covariant derivative of the difference $g_{\varphi}-g$.
To this end, we set
\begin{equation}\label{def-S}
S:=S(g_{\varphi},\,g):=\arrowvert\nabla^{g}g_{\varphi}\arrowvert^2_{g_{\varphi}}.
\end{equation}
Then from the definition of $S$, we see that
\begin{equation*}
\begin{split}
S=&g_{\varphi}^{i\bar{\jmath}}g_{\varphi}^{k\bar{l}}g_{\varphi}^{p\bar{q}}\nabla^{g}_i(g_{\varphi})_{k\bar{q}}\overline{\nabla^{g}_{j}(g_{\varphi})_{l\bar{p}}}=|\Psi|_{g_{\varphi}}^2,
\end{split}
\end{equation*}
where
\begin{equation}\label{def-Psi}
\begin{split}
\Psi_{ij}^k=\Psi_{ij}^k(g_{\varphi},\,g)&:=\Gamma(g_{\varphi})_{ij}^k-\Gamma(g)_{ij}^k=g_{\varphi}^{k\bar{l}}\nabla^{g}_i(g_{\varphi})_{j\bar{l}}.
\end{split}
\end{equation}
Define the ``half'' Laplacian with respect to the K\"ahler metric $\omega_{\varphi}$ as follows:
\begin{equation}\label{def-lap-half}
\Delta_{\omega_{\varphi},\,1/2}:=g_{\varphi}^{i\bar{\jmath}}\nabla^{g_{\varphi}}_i\nabla^{g_{\varphi}}_{\bar{\jmath}}.
\end{equation}
Then we have:
\begin{lemma}\label{lemma-S}
The function $S$ satisfies
\begin{equation*}
\begin{split}
\Delta_{\omega_{\varphi},\,X}S
&\geq |\nabla^{g_{\varphi}} \Psi|^2_{g_{\varphi}}+|\overline{\nabla}^{g_{\varphi}}\Psi|_{g_{\varphi}}^2+S-
C(n,\tr_{g}g_{\varphi},\tr_{g_{\varphi}}g)|\Rm(g)|_gS\\
&\qquad-C(n,\tr_{g}g_{\varphi},\tr_{g_{\varphi}}g))|\nabla^{g}\Rm(g)|_{g}\sqrt{S}.
\end{split}
\end{equation*}
Here, $C(\cdot,\,\cdot\,,\,\cdot)$ denotes a function that is increasing in each of its arguments
and $\overline{\nabla}^g=\nabla^{0,\,1}$. 
In particular, $\sqrt{S}$ weakly satisfies
\begin{equation*}
\begin{split}
\Delta_{\omega_{\varphi},\,X}\sqrt{S}&\geq \frac{1}{2}\sqrt{S}
-C(n,\tr_{g}g_{\varphi},\tr_{g_{\varphi}}g)|\Rm(g)|_g\sqrt{S}-C(n,\tr_{g}g_{\varphi},\tr_{g_{\varphi}}g))|\nabla^{g}\Rm(g)|_{g}.
\end{split}
\end{equation*}
\end{lemma}

\begin{proof}
We closely follow the proof of \cite[Proposition 6.9]{con-der}, which itself is based on \cite{Pho-Ses-Stu}.

Since $\varphi$ solves \eqref{cmaa}, $(M,\,g_{\varphi},\,X)$
is a shrinking gradient K\"ahler-Ricci soliton, which gives rise to a self-similar solution of the K\"ahler-Ricci flow in the following way:
if $g_{\varphi}(\tau):=(-\tau)(\varphi^{X}_{\tau})^*g_{\varphi}$ and $g(\tau):=(-\tau)(\varphi^{X}_{\tau})^*g$,
where $(\varphi^{X}_{\tau})_{\tau\,<\,0}$ is the
one-parameter family of diffeomorphisms generated by $\frac{X}{2(-\tau)}$ such that $\varphi^{X}_{\tau}\big|_{\tau\,=\,-1}=\Id_M$,
then $(g_{\varphi}(\tau))_{\tau\,<\,-1}$
is a solution of the  K\"ahler-Ricci flow with ``initial condition'' $g_{\varphi}$, i.e.,
\begin{equation*}
\begin{split}
\partial_{\tau}g_{\varphi}(\tau)\big|_{\tau\,=\,-1}&=-g_{\varphi}+\mathcal{L}_{\frac{X}{2}}g_{\varphi}\\
&=-\Ric(g_{\varphi}),\qquad g_{\varphi}(\tau)\big|_{\tau\,=\,-1}=g_{\varphi}.
 \end{split}
\end{equation*}
Similarly, by definition of the data $F$ introduced in Proposition \ref{mainprop}(i), we see that outside a compact set of $M$,
\begin{equation*}
\begin{split}
\partial_{\tau}g(\tau)\big|_{\tau\,=\,-1}&=-g+\mathcal{L}_{\frac{X}{2}}g\\
&=-\Ric(g)+\partial\bar{\partial}F,\qquad g(\tau)\big|_{\tau\,=\,-1}=g.
 \end{split}
\end{equation*}

Define $S(\tau):=S(g_{\varphi}(\tau),\,g(\tau))$, and correspondingly set $\Psi(\tau):=\Psi(g_{\varphi}(\tau),\,g(\tau))$. We adapt \cite[Proposition 3.2.8]{Bou-Eys-Gue} to our setting. By a
brute force computation, we have that
\begin{equation}
\begin{split}\label{brute-force-S}
\Delta_{\omega_{\varphi}}S&=2\Re\left(g_{\varphi}^{i\bar{\jmath}}g_{\varphi}^{p\bar{q}}(g_{\varphi})_{k\bar{l}}\left(\Delta_{\omega_{\varphi},\,1/2}\Psi_{ip}^k\right)
\overline{\Psi_{jq}^l}\right)+|\nabla^{g_{\varphi}} \Psi|^2_{g_{\varphi}}+|\overline{\nabla}^{g_{\varphi}}\Psi|_{g_{\varphi}}^2\\
&\qquad+\Ric(g_{\varphi})^{i\bar{\jmath}}g_{\varphi}^{p\bar{q}}(g_{\varphi})_{k\bar{l}}\Psi_{ip}^k\overline{\Psi_{jq}^l}
+g_{\varphi}^{i\bar{\jmath}}\Ric(g_{\varphi})^{p\bar{q}}(g_{\varphi})_{k\bar{l}}\Psi_{ip}^k\overline{\Psi_{jq}^l}-g_{\varphi}^{i\bar{\jmath}}g_{\varphi}^{p\bar{q}}\Ric(g_{\varphi})_{k\bar{l}}\Psi_{ip}^k\overline{\Psi_{jq}^l},
\end{split}
\end{equation}
where
\begin{equation*}
\begin{split}
&T^{i\bar{\jmath}}:=g_{\varphi}^{i\bar{l}}g_{\varphi}^{k\bar{\jmath}}T_{k\bar{l}},
\end{split}
\end{equation*}
for $T_{k\bar{l}}\in\Lambda^{1,\,0}M\otimes\Lambda^{0,\,1}M$. Since this flow is evolving only by diffeomorphism and scalings, we also know that
\begin{equation}
\begin{split}\label{time-der-S-I}
S(\tau)&=(-\tau)^{-1}(\varphi^{X}_{\tau})^*S(g_{\varphi},\,g),\\
\partial_{\tau}S|_{\tau\,=\,-1}&=S+\frac{X}{2}\cdot S.
\end{split}
\end{equation}
In addition, by writing $S$ in terms of the tensor $\Psi$, we have that
\begin{equation}
\begin{split}\label{time-der-S-II}
\partial_{\tau}S|_{\tau\,=\,-1}&=-\left(\frac{1}{2}\mathcal{L}_{X}g_{\varphi}-g_{\varphi}\right)^{i\bar{\jmath}}g_{\varphi}^{p\bar{q}}(g_{\varphi})_{k\bar{l}}\Psi_{ip}^k\overline{\Psi_{jq}^l}
-g_{\varphi}^{i\bar{\jmath}}\left(\frac{1}{2}\mathcal{L}_{X}g_{\varphi}-g_{\varphi}\right)^{p\bar{q}}(g_{\varphi})_{k\bar{l}}\Psi_{ip}^k\overline{\Psi_{jq}^l}\\
&\quad+g_{\varphi}^{i\bar{\jmath}}g_{\varphi}^{p\bar{q}}\left(\frac{1}{2}\mathcal{L}_{X}g_{\varphi}-g_{\varphi}\right)_{k\bar{l}}\Psi_{ip}^k\overline{\Psi_{jq}^l}+g_{\varphi}^{i\bar{\jmath}}g_{\varphi}^{p\bar{q}}(g_{\varphi})_{k\bar{l}}\partial_{\tau}\Psi(\tau)_{ip}^k|_{\tau\,=\,-1}\overline{\Psi(\tau)_{jq}^l}\\
&\quad+g_{\varphi}^{i\bar{\jmath}}g_{\varphi}^{p\bar{q}}(g_{\varphi})_{k\bar{l}}\Psi(\tau)_{ip}^k\overline{\partial_{\tau}\Psi(\tau)_{jq}^l}|_{\tau\,=\,-1}.
\end{split}
\end{equation}
By combining and \eqref{hot}, \eqref{time-der-S-I}, \eqref{time-der-S-II}, we obtain the first desired estimate, namely that
\begin{equation}\label{dream-come-true}
\begin{split}
\Delta_{\omega_{\varphi}}S-\frac{X}{2}\cdot S&\geq |\nabla^{g_{\varphi}} \Psi|^2_{g_{\varphi}}+|\overline{\nabla}^{g_{\varphi}}\Psi|_{g_{\varphi}}^2+S-C(n,\tr_{g}g_{\varphi},\tr_{g_{\varphi}}g))|\Rm(g)|_gS\\
&\qquad-C(n,\tr_{g}g_{\varphi},\tr_{g_{\varphi}}g))|\nabla^{g}\Rm(g)|_{g}\sqrt{S}
\end{split}
\end{equation}
for some positive constant $C>0$. Notice that the last three terms of the right-hand side of \eqref{brute-force-S} cancel with those of the first three terms of the right-hand side of \eqref{time-der-S-II} thanks to \eqref{hot}.

The second estimate is a straightforward consequence of the Kato inequality and \eqref{dream-come-true}.
\end{proof}

We also need the next lemma concerning the integrability of $S$.
\begin{lemma}\label{lemma-int-S}
The function $S$ lies in $L^1(e^{-f_{\varphi}}\omega_{\varphi}^n)$.
\end{lemma}

\begin{proof}
Since $X\cdot \tr_{\omega}\omega_{\varphi}=\tr_{\omega}\mathcal{L}_X\omega_{\varphi}-(\omega_{\varphi},\mathcal{L}_X\omega)_{\omega}$, adapting Yau's $C^2$-estimate as in \cite[Proposition 3.2.4]{Bou-Eys-Gue} leads to:
%\textcolor{red}{Something was up with the term $\tr_{\omega}\left(\rho(\omega_{\varphi})+\mathcal{L}_{\frac{X}{2}}\omega_{\varphi}\right)$ here. I can't remember what is was though.}
\begin{equation*}
\Delta_{\omega_{\varphi},\,X}\tr_{\omega}\omega_{\varphi}=g^{-1}\ast g_{\varphi}^{-1}\ast g_{\varphi}^{-1}\ast \nabla^gg_{\varphi}\ast \nabla^gg_{\varphi}+(\omega_{\varphi},\mathcal{L}_{\frac{X}{2}}\omega)_{\omega}-\tr_{\omega}\left(\rho(\omega_{\varphi})+\mathcal{L}_{\frac{X}{2}}\omega_{\varphi}\right)+\Rm(g)\ast g_{\varphi}^{-1}\ast g_{\varphi},
\end{equation*}
where $g^{-1}\ast g_{\varphi}^{-1}\ast g_{\varphi}^{-1}\ast \nabla^gg_{\varphi}\ast \nabla^gg_{\varphi}\geq C(n,\tr_{g}g_{\varphi},\tr_{g_{\varphi}}g)^{-1}S$ and $ \Rm(g)\ast g_{\varphi}^{-1}\ast g_{\varphi}\geq -C(n,\tr_{g}g_{\varphi},\tr_{g_{\varphi}}g) (f+C)^{-1}$ because of the quadratic curvature decay of the background metric $g$.
Invoking the soliton equation and the aforementioned lower bounds, one gets on $M$:
\begin{equation}\label{omymy-tr-yau}
\Delta_{\omega_{\varphi},\,X}\tr_{\omega}\omega_{\varphi}\geq C^{-1}S-C(f+C)^{-1},
\end{equation}
where $C>0$ is a uniform positive constant. Here we have used the equivalence of the metrics given by 
Proposition \ref{equiv-met-prop-C^2}, together with the fact that $|\mathcal{L}_{\frac{X}{2}}\omega-\omega|_g=O(f^{-1})$ by construction as dictated in Proposition \ref{mainprop}(i).

Consider next the cut-off function $\psi_R$ as defined in Lemma \ref{lemma-moments} with $R\geq 1$. Observe that thanks to \eqref{omymy-tr-yau},
\begin{equation*}\label{helpp}
\begin{split}
\int_M\psi_R\left(\Delta_{\omega_{\varphi},\,X}\tr_{\omega}\omega_{\varphi}\right)\,e^{-f_{\varphi}}\omega_{\varphi}^n&\geq C^{-1}\int_M\psi_R S\, e^{-f_{\varphi}}\omega_{\varphi}^n-C\int_M\psi_R(f+C)^{-1}\, e^{-f_{\varphi}}\omega_{\varphi}^n\\
&\geq C^{-1}\int_M\psi_R S\, e^{-f_{\varphi}}\omega_{\varphi}^n-C',
\end{split}
\end{equation*}
where $C'>0$ denotes a positive constant independent of $R\geq 1$. An integration by parts then shows that the left-hand side of \eqref{helpp} is bounded from above uniformly with respect to $R$, since $|\Delta_{\omega_{\varphi},\,X}\psi_R|$ and $\tr_{\omega}\omega_{\varphi}$ are both bounded independently of $R$. Indeed, the former because
\begin{equation*}
\begin{split}
|\Delta_{\omega_{\varphi},\,X}\psi_R|&\leq |\psi'(f_{\varphi}\cdot R^{-1})||f_{\varphi}|\cdot R^{-1}+|\psi''(f_{\varphi}\cdot R^{-1})||X|^2_{g_{\varphi}}R^{-2}\\
&\leq C,
\end{split}
\end{equation*}
where we have used the normalisation \eqref{normal-f} of $f_{\varphi}$
in the first inequality together with \eqref{equ:3}
(more specifically $|X|^2_{g_{\varphi}}\leq 2f_{\varphi}+2n$) in the second inequality.
The bound on $\tr_{\omega}\omega_{\varphi}$ is by Proposition \ref{equiv-met-prop-C^2}. This completes the proof of the lemma.   
\end{proof}

We now present:
\begin{proof}[Proof of Proposition \ref{prop-bds-C3}]
We first claim that $S$ is bounded on $M$. Indeed, define $U:=\sqrt{S}-C$ where $C>0$ is such that $\Delta_{\omega_{\varphi},\,X}U\geq  U/4$ outside a compact set, a fact ensured by Lemma \ref{lemma-S} and the boundedness of the curvature (and its covariant derivatives) of the metric $g$ and let $G(f_{\varphi}):=f_{\varphi}$ and $V:=\frac{1}{4}$. Lemma \ref{lemma-int-S} lets us apply Lemma \ref{l2-ppe-max} to get the desired intermediate result, i.e. that $S$ is bounded on $M$.

Let us now show that $S$ decays quadratically at infinity. Invoking Lemma \ref{lemma-S} again and Corollary \ref{coro-pot-bd}, one has outside a compact set of $M$,
\begin{equation*}
\begin{split}
\frac{1}{2}\Delta_{g_{\varphi},\,X}\left(f_{\varphi}S\right)&\geq (f_{\varphi} S)\left(1-\frac{C}{f}\right)-C\frac{f_{\varphi}}{f^2}+g_{\varphi}(\nabla^{g_{\varphi}}\log f_{\varphi},\nabla^{g_{\varphi}}\left(f_{\varphi}S\right))\\
&\quad+\left(\frac{1}{2}\frac{\Delta_{g_{\varphi},\,X}f_{\varphi}}{f_{\varphi}}-|\nabla^{g_{\varphi}}\log f_{\varphi}|^2_{g_{\varphi}}\right)\left(f_{\varphi}S\right)\\
&\geq -C\frac{(f_{\varphi} S)+1}{f_{\varphi}}+g_{\varphi}(\nabla^{g_{\varphi}}\log f_{\varphi},\nabla^{g_{\varphi}}\left(f_{\varphi}S\right)).
\end{split}
\end{equation*}
For $\alpha>0$, consider $U:=\left((f_{\varphi} S)+1\right)e^{\alpha f_{\varphi}^{-1}}$ and compute as follows outside a compact set of $M$ so that $f_{\varphi}>0$:
\begin{equation*}
\begin{split}
\frac{1}{2}\Delta_{g_{\varphi},\,(1+2f_{\varphi}^{-1})X}U&=\frac{e^{\alpha f_{\varphi}^{-1}}}{2}\Delta_{g_{\varphi},\,(1+2f_{\varphi}^{-1})X}\left(f_{\varphi}S\right)+\alpha g_{\varphi}(\nabla^{g_{\varphi}} f_{\varphi}^{-1},\nabla^{g_{\varphi}}U)\\
&\quad+\left((f_{\varphi} S)+1\right)\frac{1}{2}\Delta_{g_{\varphi},\,(1+2f_{\varphi}^{-1})X}e^{\alpha f_{\varphi}^{-1}}-\alpha^2|\nabla^{g_{\varphi}}f_{\varphi}^{-1}|^2_{g_{\varphi}}U\\
&\geq \left(\alpha\frac{1}{2}\Delta_{g_{\varphi},\,(1+2f_{\varphi}^{-1})X}f_{\varphi}^{-1}-Cf_{\varphi}^{-1}-\frac{\alpha^2}{2}|\nabla^{g_{\varphi}}f_{\varphi}^{-1}|^2_{g_{\varphi}}\right)U+\alpha g_{\varphi}(\nabla^{g_{\varphi}} f_{\varphi}^{-1},\nabla^{g_{\varphi}}U)\\
&\geq \left(\alpha-C-\frac{\alpha^2}{2}f_{\varphi}^{-3}|\nabla^{g_{\varphi}}f_{\varphi}|^2_{g_{\varphi}}\right)f_{\varphi}^{-1}\cdot U+\alpha g_{\varphi}(\nabla^{g_{\varphi}} f_{\varphi}^{-1},\nabla^{g_{\varphi}}U).
\end{split}
\end{equation*}
Now, choose $\alpha=2+C$ so that the previous estimate gives outside a sufficiently large compact set:
\begin{equation*}
\begin{split}
\frac{1}{2}\Delta_{g_{\varphi},\,(1+2f_{\varphi}^{-1})X}U&\geq f_{\varphi}^{-1}\cdot U+\alpha g_{\varphi}(\nabla^{g_{\varphi}} f_{\varphi}^{-1},\nabla^{g_{\varphi}}U).
\end{split}
\end{equation*}
Here we have used Lemma \ref{id-sol-egs} to handle the term $\frac{\alpha^2}{2}f_{\varphi}^{-3}|\nabla^{g_{\varphi}}f_{\varphi}|^2_{g_{\varphi}}$.

Now, $U=O(f_{\varphi})$ since we have shown in the first part of this proof that $S$ is bounded on $M$. Therefore, $U\in L^2(e^{-f_{\varphi}}\omega_{\varphi}^n)$. Applying Lemma \ref{l2-ppe-max} to $U$ defined above, $G(f_{\varphi}):=f_{\varphi}+2\log f_{\varphi}+2\alpha f_{\varphi}^{-1}$ and $V:=2f_{\varphi}^{-1}$, gives that $U$ is bounded on $M$, i.e. $S$ decays quadratically at infinity as promised.
\end{proof}
\begin{prop}[{Quadratic curvature decay of curvature}]\label{prop-quad-curv-dec}
Let $\varphi$ satisfy \eqref{cmaa}. There exists $C>0$ such that  outside a compact set of $M$, $|\Rm(g_{\varphi})|_{g_{\varphi}}\leq Cf_{\varphi}^{-1}$.
\end{prop}

\begin{proof}
The proof of this proposition follows the same lines as those of the proof of Proposition \ref{prop-bds-C3}. We will therefore be sketchy. As $(M,\omega_{\varphi},X)$ is a gradient shrinking gradient K\"ahler-Ricci soliton, the norm of the curvature tensor $\Rm(g_{\varphi})$ satisfies the following schematic  differential inequality whose proof can be found for instance in \cite{??}:
\begin{equation*}
\frac{1}{2}\Delta_{g_{\varphi},\,X}|\Rm(g_{\varphi})|^2_{g_{\varphi}}\geq |\nabla^{g_{\varphi}}\Rm(g_{\varphi})|^2_{g_{\varphi}}+2|\Rm(g_{\varphi})|^2_{g_{\varphi}}-C(n)|\Rm(g_{\varphi})|^3_{g_{\varphi}}
\end{equation*}
In particular, by Kato's inequality, the function $|\Rm(g_{\varphi})|_{g_{\varphi}}$ satisfies weakly:
\begin{equation*}
\frac{1}{2}\Delta_{g_{\varphi},\,X}|\Rm(g_{\varphi})|_{g_{\varphi}}\geq |\Rm(g_{\varphi})|_{g_{\varphi}}-C(n)|\Rm(g_{\varphi})|^2_{g_{\varphi}}.
\end{equation*}
Now, Lemma \ref{lemma-S} together with the curvature decay of the metric $g$ and Proposition \ref{equiv-met-prop-C^2} ensure that the function $S$ satisfies outside a compact set of $M$, 
\begin{equation*}
\begin{split}
\frac{1}{2}\Delta_{g_{\varphi},\,X}S
&\geq|\overline{\nabla}^{g_{\varphi}}\Psi|_{g_{\varphi}}^2+(1-Cf_{\varphi}^{-1})S-Cf_{\varphi}^{-2},
\end{split}
\end{equation*}
where we have used Corollary \ref{coro-pot-bd} to express the curvature decay of $g$ in terms of $f_{\varphi}$.

Now, $|\overline{\nabla}^{g_{\varphi}}\Psi|_{g_{\varphi}}^2=|\Rm(g_{\varphi})-\Rm(g)|^2_{g_{\varphi}}\geq \frac{1}{2}|\Rm(g_{\varphi})|^2_{g_{\varphi}}-Cf_{\varphi}^{-2}$ outside a compact set by the triangular inequality for some uniform positive constant $C$. Therefore, 
\begin{equation}
\begin{split}\label{S-rm-ineq}
\frac{1}{2}\Delta_{g_{\varphi},\,X}S
&\geq \frac{1}{2}|\Rm(g_{\varphi})|^2_{g_{\varphi}}+(1-Cf_{\varphi}^{-1})S-Cf_{\varphi}^{-2}.
\end{split}
\end{equation}
From \eqref{S-rm-ineq}, we can already show that $|\Rm(g_{\varphi})|_{g_{\varphi}}\in L^2(e^{-f_{\varphi}}\omega_{\varphi}^n)$ in the same spirit of the proof of Lemma \ref{lemma-int-S}. Then by considering a linear combination of $S$ and $|\Rm(g_{\varphi})|_{g_{\varphi}}$, one is led to the following differential inequality for $A>0$:
\begin{equation}
\begin{split}\label{S-Rm-diff-A}
\frac{1}{2}\Delta_{g_{\varphi},\,X}\left(AS+|\Rm(g_{\varphi})|_{g_{\varphi}}\right)
&\geq \left(\frac{A}{2}-C(n)\right)|\Rm(g_{\varphi})|^2_{g_{\varphi}}+|\Rm(g_{\varphi})|_{g_{\varphi}}+A(1-Cf_{\varphi}^{-1})S-ACf_{\varphi}^{-2}\\
&\geq (1-Cf_{\varphi}^{-1}) \left(AS+|\Rm(g_{\varphi})|_{g_{\varphi}}\right)-ACf_{\varphi}^{-2},
\end{split}
\end{equation}
provided $A\geq 2C(n)$. From \eqref{S-Rm-diff-A}, one can deduce that $|\Rm(g_{\varphi})|_{g_{\varphi}}$ is bounded on $M$ in the same way we proved that $S$ was bounded on $M$ in the proof of Proposition \ref{prop-bds-C3}. Then one can show that $g_{\varphi}$ has quadratic curvature decay by deriving a differential inequality for an auxiliary function of the form $e^{\alpha f_{\varphi}^{-1}}f_{\varphi}\left(AS+|\Rm(g_{\varphi})|_{g_{\varphi}}\right)$ with $\alpha>0$ suitably chosen.
\end{proof}

\begin{corollary}[{Quadratic curvature decay of curvature with derivatives}]\label{coro-shi-est}
For $k\geq 0$, there exists $C_k>0$ such that $|\nabla^{g_{\varphi},\,k}\Rm(g_{\varphi})|_{g_{\varphi}}\leq C_kf_{\varphi}^{-1-\frac{k}{2}}$ outside a $k$-independent compact set of $M$.
\end{corollary}
 
 The proof of Corollary \ref{coro-shi-est} is a direct consequence of Shi's estimates for covariant derivatives of the curvature combined with Proposition \ref{prop-quad-curv-dec}.

\bibliographystyle{amsalpha}

\bibliography{ref2}

\def\cprime{$'$} \def\cprime{$'$}
\providecommand{\bysame}{\leavevmode\hbox to3em{\hrulefill}\thinspace}
\providecommand{\MR}{\relax\ifhmode\unskip\space\fi MR }
\providecommand{\MRhref}[2]{%
  \href{http://www.ams.org/mathscinet-getitem?mr=#1}{#2}
}
\providecommand{\href}[2]{#2}
\begin{thebibliography}{EMT11}

\bibitem[AC25]{cy4}
V.~Apostolov and C.~Cifarelli, \emph{Hamiltonian 2-forms and new explicit
  {C}alabi-{Y}au metrics and gradient steady {K}\"ahler-{R}icci solitons on
  {$\Bbb{C}^n$}}, J. Differential Geom. \textbf{130} (2025), no.~3, 517--570.
  \MR{4918214}

\bibitem[BCD25]{babu}
I.~Babu, R.~Conlon, and A.~Deruelle, \emph{An {A}ubin continuity path for
  asymptotically conical toric shrinking gradient {K}\"ahler-{R}icci solitons:
  openness and a solution for $t=0$}, arXiv:2512.18137 (2025).

\bibitem[BEG13]{Bou-Eys-Gue}
S.~Boucksom, P.~Eyssidieux, and V.~Guedj, \emph{An introduction to the
  {K}\"ahler-{R}icci flow}, Lecture Notes in Mathematics, vol. 2086, Springer,
  Cham, 2013, pp.~viii+333. \MR{3202578}

\bibitem[BG08]{book:Boyer}
C.~Boyer and K.~Galicki, \emph{Sasakian geometry}, Oxford Mathematical
  Monographs, Oxford University Press, Oxford, 2008. \MR{2382957}

\bibitem[BM24]{heather}
O.~Biquard and H.~Macbeth, \emph{Steady {K}\"ahler-{R}icci solitons on crepant
  resolutions of finite quotients of {$\Bbb C^n$}}, J. Lond. Math. Soc. (2)
  \textbf{109} (2024), no.~1, Paper No. e12833, 27. \MR{4680223}

\bibitem[BN90]{SCV6}
W.~Barth and R.~Narasimhan (eds.), \emph{Several complex variables. {VI}},
  Encyclopaedia of Mathematical Sciences, vol.~69, Springer-Verlag, Berlin,
  1990, Complex manifolds, A translation of Sovremennye problemy matematiki.
  Fundamentalnye napravleniya, Tom 69, Akad. Nauk SSSR, Vsesoyuz. Inst. Nauchn.
  i Tekhn. Inform., Moscow. \MR{1095088 (91i:32001)}

\bibitem[Bry08]{Bry-Kah-Sol}
R.~L. Bryant, \emph{Gradient {K}\"ahler {R}icci solitons}, Ast\'erisque (2008),
  no.~321, 51--97, G{\'e}om{\'e}trie diff{\'e}rentielle, physique
  math{\'e}matique, math{\'e}matiques et soci{\'e}t{\'e}. I. \MR{2521644}

\bibitem[Cao96]{Cao-KR-sol}
H.-D. Cao, \emph{Existence of gradient {K}\"ahler-{R}icci solitons}, Elliptic
  and parabolic methods in geometry ({M}inneapolis, {MN}, 1994), A K Peters,
  Wellesley, MA, 1996, pp.~1--16. \MR{1417944}

\bibitem[CD20a]{con-der}
R.~J. Conlon and A.~Deruelle, \emph{Expanding {K}\"{a}hler-{R}icci solitons
  coming out of {K}\"{a}hler cones}, J. Differential Geom. \textbf{115} (2020),
  no.~2, 303--365. \MR{4100705}

\bibitem[CD20b]{conlon33}
\bysame, \emph{Steady gradient {K}\"ahler-{R}icci solitons on crepant
  resolutions of {C}alabi-{Y}au cones}, to appear in Mem.~Amer.~Math.~Soc.,
  arXiv:2006.03100 (2020).

\bibitem[CDS24]{cds}
R.~Conlon, A.~Deruelle, and S.~Sun, \emph{Classification results for expanding
  and shrinking gradient {K}\"ahler-{R}icci solitons}, Geom. Topol. \textbf{28}
  (2024), no.~1, 267--351. \MR{4711837}

\bibitem[CE25a]{Esp-charlie}
C.~Cifarelli and C.~Esparza, \emph{K-polystability of {A}symptotically
  {C}onical {K}\"ahler-{R}icci {S}hrinkers}, arXiv:2512.03323 (2025).

\bibitem[CE25b]{cif-esp}
\bysame, \emph{K-polystability of {A}symptotically {C}onical {K}\"ahler-{R}icci
  {S}hrinkers}, arXiv:2512.03323 (2025).

\bibitem[CH13]{Conlon}
R.~J. Conlon and H.-J. Hein, \emph{Asymptotically conical {C}alabi-{Y}au
  manifolds, {I}}, Duke Math. J. \textbf{162} (2013), no.~15, 2855--2902.
  \MR{3161306}

\bibitem[Cif20]{charlie}
C.~Cifarelli, \emph{Uniqueness of shrinking gradient {K}\"ahler-{R}icci
  solitons on non-compact toric manifolds}, to appear in J.~Lond.~Math.~Soc.,
  arXiv:2010.00166v3 (2020).

\bibitem[Cif24]{charles1}
\bysame, \emph{Explicit complete {R}icci-flat metrics and {K}\"ahler-{R}icci
  solitons on direct sum bundles}, arXiv:2410.23645 (2024).

\bibitem[CLN06]{Cho-Lu-Ni-Boo}
Bennett Chow, Peng Lu, and Lei Ni, \emph{Hamilton's {R}icci flow}, Graduate
  Studies in Mathematics, vol.~77, American Mathematical Society, Providence,
  RI, 2006. \MR{2274812 (2008a:53068)}

\bibitem[CR21]{cy2}
R.~Conlon and F.~Rochon, \emph{New examples of complete {C}alabi-{Y}au metrics
  on {$\Bbb C^n$} for {$n \geq 3$}}, Ann. Sci. \'Ec. Norm. Sup\'er. (4)
  \textbf{54} (2021), no.~2, 259--303. \MR{4258163}

\bibitem[CS18]{collinss}
T.~Collins and G.~Sz\'{e}kelyhidi, \emph{K-semistability for irregular
  {S}asakian manifolds}, J.~Differential~Geom. \textbf{109} (2018), no.~1,
  81--109. \MR{3798716}

\bibitem[CZ10]{caoo}
H.-D. Cao and D.~Zhou, \emph{On complete gradient shrinking {R}icci solitons},
  J.~Differential~Geom. \textbf{85} (2010), no.~2, 175--185. \MR{2732975}

\bibitem[DW11]{Wang}
A.~Dancer and M.~Wang, \emph{On {R}icci solitons of cohomogeneity one}, Ann.
  Global Anal. Geom. \textbf{39} (2011), no.~3, 259--292. \MR{2769300
  (2012a:53124)}

\bibitem[EMT11]{topping}
J.~Enders, R.~M\"uller, and P.~Topping, \emph{On type-{I} singularities in
  {R}icci flow}, Comm. Anal. Geom. \textbf{19} (2011), no.~5, 905--922.
  \MR{2886712}

\bibitem[Esp25a]{Esp-sc}
C.~Esparza, \emph{Shrinking gradient {K}\"ahler-{R}icci solitons are
  simply-connected}, arXiv:2503.05838 (2025).

\bibitem[Esp25b]{Esp25}
\bysame, \emph{Uniqueness of asymptotically conical shrinking gradient
  {K}\"ahler-{R}icci solitons}, arXiv:2502.13521 (2025).

\bibitem[FIK03]{FIK}
M.~Feldman, T.~Ilmanen, and D.~Knopf, \emph{Rotationally symmetric shrinking
  and expanding gradient {K}\"ahler-{R}icci solitons}, J.~Differential~Geom.
  \textbf{65} (2003), no.~2, 169--209. \MR{2058261}

\bibitem[Fut88]{fut2}
A.~Futaki, \emph{K\"ahler-{E}instein metrics and integral invariants}, Lecture
  Notes in Mathematics, vol. 1314, Springer-Verlag, Berlin, 1988. \MR{947341}

\bibitem[Fut21]{futaki3}
\bysame, \emph{Irregular {E}guchi-{H}anson type metrics and their soliton
  analogues}, Pure Appl. Math. Q. \textbf{17} (2021), no.~1, 27--53.
  \MR{4257580}

\bibitem[FW11]{futaki-wang}
A.~Futaki and M.-T. Wang, \emph{Constructing {K}\"ahler-{R}icci solitons from
  {S}asaki-{E}instein manifolds}, Asian J. Math. \textbf{15} (2011), no.~1,
  33--52. \MR{2786464}

\bibitem[Gra62]{Grau:62}
H.~Grauert, \emph{\"{U}ber {M}odifikationen und exzeptionelle analytische
  {M}engen}, Math. Ann. \textbf{146} (1962), 331--368. \MR{0137127 (25 \#583)}

\bibitem[HS16]{frank}
W.~He and S.~Sun, \emph{Frankel conjecture and {S}asaki geometry}, Adv. Math.
  \textbf{291} (2016), 912--960. \MR{3459033}

\bibitem[JMR16]{Jef-Maz-Rub}
T.~Jeffres, R.~Mazzeo, and Y.~Rubinstein, \emph{K\"ahler-{E}instein metrics
  with edge singularities}, Ann. of Math. (2) \textbf{183} (2016), no.~1,
  95--176, With appendices by Chi Li and Rubinstein. \MR{3432582}

\bibitem[Kol07]{kollar}
J.~Koll{\'a}r, \emph{Lectures on resolution of singularities}, Annals of
  Mathematics Studies, vol. 166, Princeton University Press, Princeton, NJ,
  2007. \MR{2289519}

\bibitem[Laz04]{lazarfeld}
R.~Lazarsfeld, \emph{Positivity in algebraic geometry. {I}}, Ergebnisse der
  Mathematik und ihrer Grenzgebiete. 3. Folge. A Series of Modern Surveys in
  Mathematics [Results in Mathematics and Related Areas. 3rd Series. A Series
  of Modern Surveys in Mathematics], vol.~48, Springer-Verlag, Berlin, 2004,
  Classical setting: line bundles and linear series. \MR{2095471}

\bibitem[Li10]{chili}
C.~Li, \emph{On rotationally symmetric {K}\"ahler-{R}icci solitons},
  arXiv:1004.4049 (2010).

\bibitem[Li19]{cy1}
Y.~Li, \emph{A new complete {C}alabi-{Y}au metric on {$\Bbb C^3$}}, Invent.
  Math. \textbf{217} (2019), no.~1, 1--34. \MR{3958789}

\bibitem[Li23]{cy5}
\bysame, \emph{S{YZ} geometry for {C}alabi-{Y}au 3-folds: {T}aub-{NUT} and
  {O}oguri-{V}afa type metrics}, Mem. Amer. Math. Soc. \textbf{292} (2023),
  no.~1453, v+126. \MR{4679706}

\bibitem[Ma26]{cy6}
T.~Ma, \emph{New {C}alabi-{Y}au {M}etrics of {T}aub-{N}{U}{T} {T}ype on
  $\mathbb{C}^{N+ 1}$}, arXiv:2601.06756 (2026).

\bibitem[MSY06]{MSY}
Dario Martelli, James Sparks, and Shing-Tung Yau, \emph{The geometric dual of
  {$a$}-maximisation for toric {S}asaki-{E}instein manifolds}, Comm. Math.
  Phys. \textbf{268} (2006), no.~1, 39--65. \MR{2249795}

\bibitem[MW15]{munteanu}
O.~Munteanu and J.~Wang, \emph{Topology of {K}\"ahler {R}icci solitons},
  J.~Differential~Geom. \textbf{100} (2015), no.~1, 109--128. \MR{3326575}

\bibitem[MW17]{munty}
\bysame, \emph{Conical structure for shrinking {R}icci solitons}, J. Eur. Math.
  Soc. (JEMS) \textbf{19} (2017), no.~11, 3377--3390. \MR{3713043}

\bibitem[Nab10]{naber}
A.~Naber, \emph{Noncompact shrinking four solitons with nonnegative curvature},
  J. Reine Angew. Math. \textbf{645} (2010), 125--153. \MR{2673425}

\bibitem[PSS07]{Pho-Ses-Stu}
D.~Phong, N.~Sesum, and J.~Sturm, \emph{Multiplier ideal sheaves and the
  {K}\"ahler-{R}icci flow}, Comm. Anal. Geom. \textbf{15} (2007), no.~3,
  613--632. \MR{2379807}

\bibitem[Ros63]{Rossi2}
H.~Rossi, \emph{Vector fields on analytic spaces}, Ann. of Math. (2)
  \textbf{78} (1963), 455--467. \MR{0162973 (29 \#277)}

\bibitem[Sch21]{schafer2}
J.~Sch{\"a}fer, \emph{Asymptotically cylindrical steady {K}{\"a}hler-{R}icci
  solitons}, arXiv:2103.12629 (2021).

\bibitem[Sch23]{schaffer}
J.~Sch\"afer, \emph{Existence and uniqueness of {$S^1$}-invariant
  {K}\"ahler-{R}icci solitons}, Ann. Fac. Sci. Toulouse Math. (6) \textbf{32}
  (2023), no.~1, 15--53. \MR{4574737}

\bibitem[SZ24]{JunshengSong}
S.~Sun and J.~Zhang, \emph{{K}\"ahler-{R}icci shrinkers and {F}ano fibrations},
  arXiv:2410.09661 (2024).

\bibitem[Sz{\'e}20]{cy3}
G.~Sz{\'e}kelyhidi, \emph{Uniqueness of some {C}alabi-{Y}au metrics on {${\bf
  C}^n$}}, Geom. Funct. Anal. \textbf{30} (2020), no.~4, 1152--1182.
  \MR{4153912}

\bibitem[vC11]{vanC4}
C.~van Coevering, \emph{Examples of asymptotically conical {R}icci-flat
  {K}\"ahler manifolds}, Math. Z. \textbf{267} (2011), 465--496.

\bibitem[WL23]{wangli}
B.~Wang and Y.~Li, \emph{On {K}{\"a}hler {R}icci shrinker surfaces},
  arXiv:2301.09784, to appear in Acta Math. (2023).

\bibitem[WZ04]{soliton}
X.-J. Wang and X.~Zhu, \emph{K\"ahler-{R}icci solitons on toric manifolds with
  positive first {C}hern class}, Adv. Math. \textbf{188} (2004), no.~1,
  87--103. \MR{2084775 (2005d:53074)}

\bibitem[Yau78]{Calabiconj}
S.-T. Yau, \emph{On the {R}icci curvature of a compact {K}\"ahler manifold and
  the complex {M}onge-{A}mp\`ere equation. {I}}, Comm. Pure Appl. Math.
  \textbf{31} (1978), no.~3, 339--411. \MR{480350}

\bibitem[Zha09]{zhang12}
Z.-H. Zhang, \emph{On the completeness of gradient {R}icci solitons}, Proc.
  Amer. Math. Soc. \textbf{137} (2009), no.~8, 2755--2759. \MR{2497489}

\end{thebibliography}

\end{document}